\documentclass[english, 11pt]{amsart}
\usepackage{amssymb}
\usepackage{amsfonts}
\usepackage{amscd}
\usepackage[T1]{fontenc}
\usepackage{color}
\usepackage{tikz}
\usepackage{changes}
\usepackage{cancel}

\usepackage{mathrsfs} 

\def\deg{\operatorname{deg}}
\def\ac{{\overline{\rm ac}}}
\def\Supp{\operatorname{Supp}}

\def\moi{\operatorname{moi}}
\def\lct{\operatorname{lct}}
\def\Gal{\operatorname{Gal}}
\def\Sing{\operatorname{Sing}}
\def\Tr{\operatorname{Tr}}
\def\Frob{\operatorname{Frob}}
\def\H{\operatorname{H}}

\def\11{{\mathbf 1}}
\def\AA{{\mathbb A}}

\def\CC{{\mathbb C}}

\def\FF{{\mathbb F}}

\def\NN{{\mathbb N}}

\def\QQ{{\mathbb Q}}

\def\ZZ{{\mathbb Z}}

\def\cA{{\mathcal A}}

\def\cL{{\mathcal L}}
\def\cM{{\mathcal M}}

\def\cO{{\mathcal O}}

\def\cX{{\mathcal X}}
\def\cY{{\mathcal Y}}

\newcounter{dummy} \numberwithin{dummy}{section}

\newtheorem{thm}[dummy]{Theorem}
\newtheorem{lem}[dummy]{Lemma}
\newtheorem{cor}[dummy]{Corollary}
\newtheorem{prop}[dummy]{Proposition}

\newtheorem{conj}[dummy]{Conjecture}

\theoremstyle{definition}
\newtheorem{defn}[dummy]{Definition}
\newtheorem{notation}[dummy]{Notation}

\newtheorem{def-prop}[dummy]{Proposition-Definition}
\newtheorem{def-theorem}[dummy]{Theorem-Definition}
\newtheorem{def-lem}[dummy]{Lemma-Definition}

\theoremstyle{remark}
\newtheorem{remark}[dummy]{Remark}

\theoremstyle{plain}

\numberwithin{equation}{subsection}

  {\par\medskip\noindent #1\par\begingroup%
    \advance\leftskip by 1em\advance\rightskip by 1em}%
  {\par\endgroup}

\DeclareMathOperator*{\Spec}{Spec}

\newcommand{\ord}{\operatorname{ord}}

\def\VG{\mathrm{VG}}
\newcommand{\RF}{{\rm RF}}

\begin{document}

\author{Kien Huu Nguyen}
\author{Willem Veys}
\address{K. H. Nguyen and W. Veys, KU Leuven, Department of Mathematics, Celestijnenlaan 200B, 3001 Heverlee, Belgium}
\email{kien.nguyenhuu@kuleuven.be, wim.veys@kuleuven.be}



\thanks{\fontsize{8pt}{9pt}\selectfont  The first author is supported by the Fund for Scientific Research - Flanders (Belgium) (FWO) 12X3519N. The second author is partially supported by KU Leuven grant C14/17/083. The authors want to thank Santiago Encinas and Orlando Villamayor for useful discussions.}

\title[On the motivic oscillation index]
{\fontsize{10pt}{11pt}\selectfont On the motivic oscillation index and bound of exponential sums modulo $p^m$ via analytic isomorphisms}

\begin{abstract} Let $f$ be a  polynomial in $n$ variables over some number field and $Z$ a subscheme of affine $n$-space. The notion of motivic oscillation index of $f$ at $Z$  was initiated by Cluckers (2008) and Cluckers-Musta{\c{t}}{\v{a}}-Nguyen (2019). In this paper we elaborate on this notion and raise several questions. The first one is stability                                                                                                                                                                                                                                                                                                                                                                                                                                                     under base field extension; this question is linked to a deep understanding of the density of non-archimedean local fields over which Igusa's local zeta functions of $f$ has a pole with given real part. The second one is around Igusa's conjecture for exponential sums with bounds in terms of the motivic oscillation index.
Thirdly, we wonder if the above questions only depend on the analytic isomorphism class of singularities. 
By using various techniques as the GAGA theorem, resolution of singularities and model theory,
we can answer the third question up to a base field extension. Next, by using a transfer principle between non-archimedean local fields of characteristic zero and positive characteristic, we can link all three questions with a conjecture on weights of $\ell$-adic cohomology groups of Artin-Schreier sheaves associated to jet polynomials. This way, we can answer all questions positively if $f$ is a polynomial \lq of Thom-Sebastiani type\rq\ with non-rational singularities. As a consequence, we prove Igusa's conjecture for arbitrary polynomials in three variables and polynomials with singularities of $A-D-E$ type. In an appendix, we answer affirmatively a recent question of Cluckers-Musta{\c{t}}{\v{a}}-Nguyen (2019) on poles of twisted Igusa's local zeta function of maximal order.
\end{abstract}

\maketitle

\section{Introduction}
Let $f$ be a non-zero polynomial in $\ZZ[x_1,\dots,x_n]$.
For an integer $N>1$, the exponential sum modulo $N$ for $f$ is
$$E_{f,N}:= \dfrac{1}{N}\sum_{x\in (\ZZ/N\ZZ)^n}\exp\left(\frac{2\pi if(x)}{N}\right).$$
Good estimates for such exponential sums have many applications in mathematics. If we write
$1/N=\sum_{1\leq i\leq k}a_i/p_i^{m_i}$,
where $p_i$ is prime, $a_i$ and $m_i$ are integers such that $a_i\neq 0 \mod p_i$ and $m_i>0$ for all $1\leq i\leq k$ (and all $p_i$ are different), then we have
$$E_{f,N}=\dfrac{1}{N}\prod_{i=1}^k E_{a_if,p_i^{m_i}},$$
by using the Chinese remainder theorem. So we can relate the size of $E_{f,N}$ with the size of the local factors $E_{a_if,p_i^{m_i}}$. But in order to give a good estimate for $E_{f,N}$, a uniform estimate for all local factors is still needed. In a series of papers \cite{IgusaInvent,IgusaNago,IgusaNagoya,IgusaCrell,IgusaTokyo,IgusaKyoto,Igusalecture}, Igusa studied the local factor $E_{f,p^m}$ by using as follows a $p$-adic integral.

We fix a prime number $p$ and denote by $\psi_m$ the additive character of $\QQ_p$ given by $\psi_m(x):=\exp(2\pi i x'/p^m)$ with $x'\in (x+p^m\ZZ_p)\cap\ZZ[\frac{1}{p}]$ (which is independent of the choice of $x'$).  Note that $\psi_m$ has conductor $m$, that is, it is trivial on $p^m\ZZ_p$ but non-trivial on $p^{m-1}\ZZ_p$.
 Let $\Phi$ be a Schwartz-Bruhat function on $\QQ_p^n$, that is, a locally constant function with compact support. For a positive integer $m$, the exponential sum modulo $p^m$ associated with $\Phi$ and $f$ is
\begin{equation}\label{exponential integral intro}
E_{f,\QQ_p,\Phi,\psi_m}:=\int_{\QQ_p^n}\Phi(x)\psi_m\left(f(x)\right)|dx|=\int_{\QQ_p^n}\Phi(x)\exp\left(\dfrac{2\pi i f(x)}{p^m}\right)|dx|,
\end{equation}
where $|dx|$ is the normalized Haar measure on $\QQ_p^n$ such that the volume of $\ZZ_p^n$ is $1$. In particular, if $\Phi=\11_{\ZZ_p^n}$, then $E_{f,\QQ_p,\Phi,\psi_m}=E_{f,p^m}$.
Using resolution of singularities, Igusa showed that there exists a constant $\sigma_{f,p,\Phi}=\sigma_{f,\QQ_p,\Phi}\leq +\infty$ which is maximal with respect to the following property:

\noindent
{\em for all real numbers $\sigma<\sigma_{f,p,\Phi}$, we can find a positive constant $c_p$,
depending on $f,p,\Phi,\sigma$ and satisfying
\begin{equation}\label{Igusa bound}
|E_{f,\QQ_p,\Phi,\psi_m}|\leq c_p p^{-m\sigma}
\end{equation}
for all $m\geq 1$.} In fact,  $\sigma_{f,p,\Phi}$ can be computed by poles of $p$-adic Igusa zeta functions associated with $f$ and $\Phi$.
We first recall this notion.

\medskip
We fix an angular component map $ac: \QQ_p^*\to \ZZ_p^*$, being a group homomorphism such that $ac|_{\ZZ_p^*}$ is the identity. Such a map arises by taking a uniformizing parameter $\pi$ for $\ZZ_p$ and putting $ac(x):=x\pi^{-\ord_p(x)}$.

 Let $\chi$ be a multiplicative character of $\ZZ_p^{*}$, i.e., a group homomorphism $\chi:\ZZ_p^{*}\to\CC^{*}$ with finite image. We extend it to $\QQ_p^*$ by setting $\chi:=\chi\circ ac$, using the same notation, and further to $\QQ_p$ by taking $\chi(0):=0$, and call it then a multiplicative character of $\QQ_p$. To these data one associates the Igusa local zeta function
\begin{equation}\label{Igusa zeta function intro}
Z(f,\chi,\QQ_p,\Phi;s):=\int_{\QQ_p^n}\Phi(x)\chi\bigl(ac(f(x))\bigr)|f(x)|^s|dx|.
\end{equation}

Igusa showed that $-\sigma_{f,p,\Phi}$ is either $-\infty$ or the real part of a pole of $Z(f-c,\chi,\QQ_p,\Phi;s)$ for some $\chi$ and some critical value $c$ of $f$. Moreover, all poles of these zeta functions are induced by numerical data associated with a fixed embedded resolution of $\{f-c=0\}$ (see Section \ref{section2} for more details). In fact, it is more conceptual to consider in (\ref{Igusa bound})  a \lq universal\rq\ constant $\sigma_f$, only depending on $f$, and a constant $c_p$, depending on $f$, $\sigma_f$ and $p$.   Igusa conjectured essentially that then one can take such $c_p$ independent of $p$.




\smallskip
Both the integral (\ref{exponential integral intro}) and the zeta function (\ref{Igusa zeta function intro}) can be generalized to polynomials $f$ over some subring $\cO$ of a number field (instead of $\ZZ$). We refer to Section \ref{section2} for all details; here we mention the main ideas.
Take such a ring $\cO$, which is a finitely generated $\ZZ$-algebra. We consider $p$-adic fields $L$, with valuation ring $\cO_L$, maximal ideal $\cM_L$, 
and cardinality $q_L$ of the residue field,
such that $\cO \subset \cO_L$. To non-trivial additive characters $\psi$ of $L$ and Schwartz-Bruhat functions $\Phi$ on $L^n$, we associate analogous integrals $E_{f,L,\Phi,\psi}$.
More precisely, we consider \lq  motivic\rq\ Schwartz-Bruhat functions $\Phi$, associated to a subscheme $Z$ of $\AA_{\cO}^n$, i.e., $\Phi:=\Phi_Z$ is the collection $(\Phi_{L,Z})_{L}$ with $\Phi_{L,Z} = \11_{\{x\in\cO_L^n|\overline{x}\in Z(k_L)\}}$, namely the indicator function of the \lq $\cM_L^n$-neighbourhood of the $k_L$-points of $Z$\rq.
We then denote the corresponding integrals by $E_{f,L,Z,\psi}$.

For multiplicate characters $\chi$  of $\cO_L^{*}$, we consider similarly zeta functions $Z(f,\chi,L,\Phi;s)$ and especially $Z(f,\chi,L,Z;s)$.
The motivic oscillation index of $f$ at $Z$ over a number field $K$, denoted by $\moi_K(f,Z)$  is roughly the minimum of the set $\{-\Re(s_0)\}$, where $s_0$ runs over all non-trivial poles of the zeta functions
$Z(f,\chi,L,Z;s)$, for all $L$ containing $K$ and with big enough residue field characteristic (see the precise definition in Section \ref{section2}).

Elaborating on Igusa's original conjecture, the following formulation was implicitly stated in \cite{CMN}.


\begin{conj}\label{Conjexp}Let $f$ be a non-constant polynomial in $\cO[x_1,\dots,x_n]$ and  $Z$  a subscheme of $\AA^n_\cO$. Let $K$ be a number field containing $\cO$. Then, for all $\sigma<\moi_{K}(f,Z)$, there exist a constant $c_{f,\sigma}$ and a number $M$ such that
$$|E_{f,L,Z,\psi}|\leq c_{f,\sigma} q_L^{-m_{\psi}\sigma}$$
for all local fields $L$ containing $K$, and with residue field characteristic at least $M$, and all characters $\psi$ with conductor $m_\psi\geq 2$.
\end{conj}

In case $f$ has non-rational singularities in every open neighbourhood of $Z(\CC)$, Conjecture \ref{Conjexp} was proven in \cite{CMN}.
Actually, the authors show the analogue of Conjecture \ref{Conjexp}, with $\moi_{K}(f,Z)$ replaced by the log canonical threshold $\lct_{Z}(f)$. More precisely, they prove an equivalence between that analogous conjecture and a certain upper bound for the log canonical threshold, and they prove that upper bound using techniques from the Minimal Model Program. (For the special case of two variables, that upper bound was proven independently in \cite{Veys} by elementary techniques.)
Finally, they establish that $\lct_{Z}(f)=\moi_{K}(f,Z)$ precisely when $f$ has non-rational singularities in every open neighborhood of $Z(\CC)$.
Some variants of conjecture \ref{Conjexp} can be found in \cite{Denef-Sperber, CluckersDuke, CluckersTAMS, WouterKien, Lich1, Lich2, Wright}. We refer to \cite{CMN} for more details and background on Igusa's original motivation.


Among the results in this paper, we prove Conjecture \ref{Conjexp} for all polynomials in three variables. We derive this from some general results of independent interest, for germs of polynomials defined over a number field, that are equivalent under a biholomorphic transformation. We first fix notation, in order to state these results.

\begin{defn}
Let $f$ and $g$ be non-constant polynomials in $n$ variables with coefficients in a number field $K$. Take $P,Q\in K^n$ satisfying $f(P)=g(Q)=0$. Suppose that there exist open neighbourhoods $U$ of $P$ and $V$ of $Q$ in $\CC^n$, respectively, and a biholomorphic function $\theta: U\to V$ and an invertible holomorphic function $v:V\to \CC$ such that $f=(g.v)\circ\theta$ and $\theta(P)=Q$.
 We say that {\em $(f,P)$ and $(g,Q)$ are similar over $K$} and denote this by $(f,P)\sim_K (g,Q)$.
\end{defn}

In Section \ref{section3}, we use the GAGA theorem from \cite{SerreGAGA} and model theory of algebraically closed field of characteristic $0$ to prove our first result.

\begin{thm}[Main theorem 1]\label{mainthm1} Let $f$ and $g$ be non-constant polynomials in $n$ variables with coefficients in a number field $K$, and let $P,Q\in K^n$ satisfy $f(P)=g(Q)=0$. Suppose that $(f,P)$ and $(g,Q)$ are similar over $K$. Then there exists a natural number $M$ and a finite field extension $K'$ of $K$, such that for all non-archimedean local fields $L$ containing $K'$ and with residue field characteristic at least $M$ we have
$$Z(f,\chi,L,\Phi_{L,P},s)=Z(g,\chi,L,\Phi_{L,Q},s)$$
for all multiplicative characters $\chi:L\to \CC^*$, and
 $$E_{f,L,P,\psi}=E_{g,L,Q,\psi}$$
 for all additive character $\psi$ of conductor at least $1$.
\end{thm}

We would like to conclude from this that then $\moi_K(f,P)=\moi_K(g,Q)$, in order to show that Conjecture \ref{Conjexp} holds for $(f,P)$ if and only if it holds for $(g,Q)$. We are however obstructed by an intriguing question about the motivic oscillation index.

When $Z$ is a point $P$, the notion $\moi_K(f,P)$ is an analogue in the ad\`elic context of the notions oscillation index and complex oscillation index of $f$ at $P$ in \cite{Arnold.G.V.II}.
 By construction, the motivic oscillation index seems to be defined over the algebraic closure of $K$ rather than over $K$ itself. On the other hand, if $f$ has  non-rational singularities in every open neighborhood of $Z(\CC)$, then, for all finite extensions $K'$ of $K$, we have that $\moi_{K}(f,Z)=\moi_{K'}(f,Z)$ is the log canonical threshold of $f$ at $Z$.
Hence we may expect that the following conjecture holds.

\begin{conj}\label{Conj1}  If $f$ is a non-constant polynomial in $n$ variables  with coefficients in a number field $K$ and $Z$ a subscheme of $\AA^n_K$,  then, for any finite extension $K'$ of $K$, we have that $\moi_K(f,Z)=\moi_{K'}(f,Z)$.
\end{conj}

It is well known that the \lq classical\rq\ oscillation index is in general not stable under a biholomorphic transformation (see \cite{Var} or \cite{Arnold.G.V.II}). In contrast with this fact,
we expect that the motivic oscillation index at a point is stable by biholomorphic transformations.
As a consequence of Theorem \ref{mainthm1}, by using the relation between the motivic oscillation index  and the poles of Igusa zeta functions, we show that this is implied by Conjecture \ref{Conj1}.

\begin{prop}\label{conj2} Suppose that $(f,P)$ and $(g,Q)$ are similar over $K$. If Conjecture \ref{Conj1} holds, then $\moi_K(f,P)=\moi_K(g,Q)$.
\end{prop}

In order to
handle some cases of Conjecture \ref{Conj1}, we use the transfer principle for constructible motivic functions from \cite{Cluckers-Loeser-ann} to relate exponential sums modulo $p^m$ and the motivic oscillation index  to the highest weight of $\ell$-adic \'etale cohomology groups with compact support, associated with Artin-Schreier sheaves on jet schemes over finite fields. We propose Conjecture \ref{highest} about this relation and show that this conjecture implies Conjecture \ref{Conjexp} and Conjecture \ref{Conj1}. Moreover,  by using Deligne's theory of weights on $\ell$-adic \'etale cohomology groups, we show that, if Conjecture \ref{highest} holds for $(f,P)$, then it holds for $(g,Q)$ if $(f,P)\sim_K (g,Q)$.

We prove Conjecture \ref{highest} for some cases. In particular, following from the result of \cite{CMN}, we prove that if $f$ is a polynomial \lq of Thom-Sebastiani type\rq, i.e., $f$ can be written as sum of  polynomial $g_i,1\leq i\leq r,$ where $g_i$ and $g_j$ have no common variables if $i\neq j$, and all $g_i$ have non-rational singularities, then Conjecture \ref{highest} holds for $f$. As a consequence, we have the following theorem.

\begin{thm}[Main theorem 2]\label{mainthm2} Let $(f,P)\sim_K (g,Q)$ and $g=\sum_{i=1}^r g_i$, where $g_i$ and $g_j$ have no common variables if $i\neq j$ and  $Q=Q_1\times \dots \times Q_r$. If all $g_i$ have  non-rational singularities at $Q_i$, then Conjectures  \ref{Conjexp} and  \ref{Conj1} hold for $(f,P)$.
\end{thm}

We have the following important corollaries.

\begin{cor} \label{ADE}If $f$ has singularities of type $A-D-E$ at $P$, then Conjecture \ref{Conjexp} holds for $(f,P)$.
\end{cor}

\begin{cor}\label{3variables} Conjecture \ref{Conjexp} holds for all  polynomials in at most three variables.
\end{cor}

We remark here that it is absolutely not clear how to obtain Corollaries  \ref{ADE} and \ref{3variables} from the method in \cite{CMN}, depending crucially on resolution of singularities. Indeed, we only know embedded resolution of singularities in an abstract way, in general we do not know much about the arithmetic properties of the exceptional divisors in a resolution,  the value of the motivic oscillation index is not clear in these cases, and some of these divisors  may not be essential, i.e., their numerical data $(N, \nu)$ (see Section \ref{section2}) may satisfy $\nu/N<\moi_K(f,P)$.


\medskip
We introduce the precise definitions and recall the necessary background on exponential sums and Igusa local zeta functions in Section \ref{section2}.  Then, in Section \ref{section3}, we prove Theorem \ref{mainthm1}.
In Section \ref{section4}, we develop some results about finite characteristic analogues of our integrals $E_{f,L,Z,\psi}$; we link those to our original setting in Section \ref{section5}, in particular in the context of Conjecture \ref{highest}.
We then establish our unconditional results in Section \ref{section6}, including the proof of Conjecture \ref{Conjexp} for all polynomials in three variables.

\medskip

We hope that this paper may provide some general tools to study exponential sums modulo $p^m$ and Igusa local zeta functions. Given a polynomial $f$, one may find a much more simple polynomial $g$ which is similar to $f$. If  $g$ satisfies Conjecture \ref{Conjexp}, we obtain information about Conjecture \ref{Conjexp} for $f$. So we have  room to apply the Weierstrass preparation theorem here. If the transformation is algebraic, then it is not hard to relate the integrals associated with $f$ and $g$ by using the Jacobian matrix as in Theorem \ref{mainthm1}. But even if an algebraic transformation over $\overline{\QQ}$  is found, Proposition \ref{conj2} and Theorem \ref{mainthm2} are still not clear without our work in Sections \ref{section4}, \ref{section5}, \ref{section6}. On the other hand, it is often not easy or even impossible to find an algebraic transformation.

Furthermore, we can transfer the study for exponential sums modulo $p^m$ of $f$ to the study for exponential sums of \lq jet-polynomials\rq\ of $f$ over finite fields. These jet-polynomials appear naturally with good properties,  as seen in Sections \ref{section4}, \ref{section5}, \ref{section6}. We may expect a theory about weights  related to exponential sums for these polynomials and link it with Igusa's monodromy conjecture, Bernstein-Sato polynomials, mixed hodge structures, etc. related to $f$.

\medskip
In the appendix, we answer a question in \cite{CMN} about poles of maximal order of Igusa zeta functions with non-trivial character.
More precisely, let $L$ be a $p$-adic field, $f\in L[x_1,\dots,x_n]\setminus L$, and $\Phi$ a  Schwartz-Bruhat function on $L^n$. Let also $\chi$ be a character of $\cO_L^*$ of order $d>1$.
If $s_0$ is the real part of a pole of order $n$ of $Z(f,\chi,L,\Phi;s)$, then $s_0=-\lct_{Z}(f)=-1/dk$ for some positive integer $k$.

\section{Exponential sums and Igusa local zeta functions}\label{section2}

In this section we recall for all non-archimedean local fields the notions of exponential sums modulo $p^m$ and Igusa local zeta functions, and how those are related.

\subsection{Setup, multiplicative and additive characters}

Let $L$ be a non-archimedean local field, i.e., a finite extension of $\QQ_p$ or $\FF_p((t))$ for some prime number $p$. We denote by $\cO_L$ and $\cM_L$ the valuation ring of $L$ and its maximal ideal, respectively, and by $k_L=\cO_L/\cM_L$ the residue field of $L$, of characteristic $p_L$ and cardinality $q_L$.
For $x\in L$, we denote by $\ord_L(x) \in \ZZ\cup\{+\infty\}$ its valuation and  by $|x|_L=q_L^{-\ord_L(x)}$ its absolute value.

We fix a uniformizing parameter $\pi$ for $\cO_L$, and denote by $ac:L\to\cO_L^*$ the map given by $ac(x)=x\pi^{-\ord_L(x)}$ if $x\neq 0$ and $ac(0)=0$.
For $x\in \cO_L$, we write $\overline{x}$ for the image of $x$ by the residue map $\cO_L\to \cO_L/\cM_L=k_L$, and we put $\ac(x)=\overline{ac(x)}$ if $x\in L$.

For each positive integer $n$, we endow $L^n$ with  the Haar measure $|dx|$ which is normalized such that the volume of $\cO_L^n$ is $1$.

\medskip
A {\em multiplicative character} $\chi$ of $\cO_L^*$ is a continuous homomorphism $\chi:\cO_L^*\to \CC^*$ with finite image. We also put $\chi(0)=0$.
If $\chi$ is a multiplicative character, the {\em conductor} $c(\chi)$ of $\chi$ is the smallest integer $e\geq 1$ such that $\chi$ is trivial on $1+\cM_L^e$.
An {\em additive character} of $L$ is a continuous homomorphism $\psi:L\to\CC^*$ with finite image. If $\psi$ is a non-trivial additive character, the {\em conductor} $m_{\psi}$ of $\psi$ is the integer $m$ such that $\psi$ is trivial on $\pi^m\cO_L$, but non-trivial on $\pi^{m-1}\cO_L$.

If we fix an additive character $\psi$ with conductor $0$, then any non-trivial additive character of $L$ is of the form $\psi_z:L\to\CC^{*}: x \mapsto \psi(zx)$ for some non-zero element $z$ of $L$.  In particular, we have that $m_{\psi_z}=-\ord_L(z)$.

\medskip
From now on, we mean by a non-archimedean local field $L$ implicitly a triple $(L,\pi,\psi)$, where $\pi$ is a fixed uniformizing parameter and $\psi$ is a fixed additive character  of conductor $0$.  Moreover, we will suppose that $\psi_{\pi^{-1}}$ induces the standard additive character $\Psi_{p_L}\circ \Tr_{k_L/\FF_{p_L}}: k_L\to\CC^*$ of $k_L$, where $\Psi_{p_L}:\FF_{p_L}\to\CC^*$ sends $(x\mod p_L)$ to $\exp(\frac{2\pi ix}{p_L})$ and $\Tr_{k_L/\FF_{p_L}}$ is the trace map. If needed, we will write explicitly $(L,\pi,\psi)$ instead of $L$.

\subsection{Exponential sums and Igusa local zeta functions}

\begin{defn}
Let $L$ be a non-archimedean local field, $f\in L[x_1,\dots,x_n]$ a non-constant polynomial and $\Phi$  a Schwartz-Bruhat function on $L^n$, i.e., a locally constant function with compact support.

\begin{itemize}
\item[(i)]
The {\em exponential sum} associated with $f$, $\Phi$ and an additive character $\psi_z$ is
$$E_{f,L,\Phi,\psi_z}:=\int_{L^n}\Phi(x)\psi_z(f(x))|dx|.$$
\item[(ii)] The {\em Igusa local zeta function} associated with $f$, $\Phi$ and a multiplicative character $\chi$ is
$$Z(f,\chi,L,\Phi;s):=\int_{L^n}\Phi(x)\chi\bigl(ac(f(x))\bigr)|f(x)|_L^s|dx|,$$
where $s\in \CC$ with $\Re(s)>0$.
\end{itemize}
\end{defn}
\noindent
It is well known that $Z(f,\chi,L,\Phi;s)$ is holomorphic in this region, that it has a meromorphic continuation to $\CC$, and that it is a rational function in $p^{-s}$.

We recall the relation between these objects.

\begin{prop}[\cite{DenefBourbaki}, Proposition 1.4.4]\label{zetaex}
Let $L$ be a non-archimedean local field of characteristic $0$, with a fixed uniformizing parameter $\pi$ and a fixed additive character $\psi$ of conductor $0$, and $f\in L[x_1,\dots,x_n]$ a non-constant polynomial. Let $z=u\pi^{-m}$ with  $u\in\cO_{L}^{*}$ and $m\geq 1$. Then $E_{f,L,\Phi,\psi_z}$ is equal to

\begin{align*}
Z(f,\chi_\textnormal{triv},L,\Phi;0)&+\mbox{\textnormal{Coeff}}_{t^{m-1}}\Big(\dfrac{(t-q_L)Z(f,\chi_\textnormal{triv},L,\Phi;s)}{(q_L-1)(1-t)}\Big)\\
&+\sum_{\chi\neq\chi_\textnormal{triv}}g_{\chi^{-1}}\chi(u)\mbox{\textnormal{Coeff}}_{t^{m-c(\chi)}}\big(Z(f,\chi,L,\Phi;s)\big),
\end{align*}
where $g_{\chi}$ is the Gaussian sum
\[
g_{\chi}=\frac{q_L^{1-c(\chi)}}{q_L-1}\sum_{\overline{v}\in\bigl(\cO_{L}/\cM_L^{c(\chi)}\bigr)^*}\chi(v)\psi\bigl(v/\pi^{c(\chi)}\bigr).
\]
\end{prop}

\noindent
The following lemma allows us to concentrate  on only finitely many multiplicative characters $\chi$.

\begin{lem}[\cite{Denef91,Igusalecture}]\label{zerofunction} Let $L$ be a non-archimedean local field of characteristic $0$ and $f$  a non-constant polynomial in $L[x_1,\dots,x_n]$. Suppose that $\Phi$ is a Schwartz-Bruhat function on $L^n$ satisfying $\Supp(\Phi)\cap C_f\subset f^{-1}(0)$, where $C_f$ is the critical locus of $f$.
\begin{itemize}
\item[(i)] Then $Z(f,\chi,L,\Phi;s)=0$ for all but finitely many multiplicative characters $\chi$.
\item[(ii)] Suppose moreover that $\Phi$ is residual, i.e., $\Supp(\Phi)\subset\cO_L^n$ and $\Phi(x)$ depends only on  $x\mod \cM_L$, that $\overline{f}:=f\mod \cM_L \neq 0$ and that $f$ has a resolution with tame good reduction mod $\cM_L$ (see the definition below). We consider the function $\overline{\Phi}(x \mod \cM_L):=\Phi(x)$. If  $(\Supp\overline{\Phi})\cap C_{\overline{f}}\subset \overline{f}^{-1}(0)$, then $Z(f,\chi,L,\Phi;s)=0$ for all multiplicative characters $\chi$ with conductor $c(\chi)>1$.
\end{itemize}
\end{lem}

\begin{cor}\label{invariant sigma} Let $L$ be a non-archimedean local field of characteristic $0$ and  $f\in L[x_1,\dots,x_n]$ a non-constant polynomial. Let $\Phi$ be a Schwartz-Bruhat function on $L^n$ satisfying $\Supp(\Phi)\cap C_f\subset f^{-1}(0)$. Denote by $Pol_{f,L,\Phi}$  the set of complex numbers which are poles of $Z(f,\chi,L,\Phi;s)$ with $\chi\neq\chi_\textnormal{triv}$ or poles of $(q_L^{s+1}-1)Z(f,\chi_\textnormal{triv},L,\Phi;s)$, called the {\em non-trivial poles}.  Consider the invariant
$$\sigma_{f,L,\Phi}:=\min\{-\Re(s)\mid s\in Pol_{f,L,\Phi}\}\leq +\infty,$$
where we use the convention that $\sigma_{f,L,\Phi}=+\infty$ if $Pol_{f,L,\Phi}=\emptyset$.

Then, for each $\sigma<\sigma_{f,L,\Phi}$, there exists a constant $c_{f,L,\Phi,\sigma}$ (depending on $f,L,\Phi,\sigma$) such that, for all additive character $\psi$ with conductor $m_\psi>0$, we have
$$|E_{f,L,\Phi,\psi}|\leq c_{f,L,\Phi,\sigma}q_L^{-m_\psi \sigma}.$$
Moreover, $\sigma_{f,L,\Phi}$ is the largest number with this property.
\end{cor}

\subsection{Denef's formula}\label{Denef's formula}

We recall the description of the Igusa local zeta function using resolution of singularities from \cite{Denef91}. Let $L$ be a non-archimedean local field of characteristic $0$ and $f$  a non-constant polynomial in $L[x_1,\dots,x_n]$. We set $X=\Spec L[x_1,\dots,x_n]$ and $D=\Spec L[x]/(f)$. We fix an embedded resolution $(Y,h)$ of $D$ over $L$. This means here that $Y$ is an integral smooth closed subscheme of projective space over $X$, $h:Y\rightarrow X$ is the natural map, the restriction $h:Y\backslash h^{-1}(D)\rightarrow X\backslash D$ is an isomorphism, and $(h^{-1}(D))_{\text{red}}$ has simple normal crossings as subscheme of $Y$. The existence of such $h$ follows originally from \cite[page 142, Main Theorem II]{Hironaka}.

Let $E_{i}, i\in T$, be the irreducible components of $(h^{-1}(D))_{\text{red}}$. For each $i\in T$, let $N_{i}$ be the multiplicity of $E_{i}$ in the divisor of $f\circ h$ on $Y$ and let $\nu_{i}-1$ be the multiplicity of $E_{i}$ in the divisor of $h^{*}(dx_{1}\wedge\ldots\wedge dx_{n})$. The $(N_{i},\nu_{i})_{i \in T}$ are called the \emph{numerical data} of the resolution $(Y,h)$. For each subset $I\subset T$, we consider the schemes
\[
E_{I}:=\cap_{i\in I} E_{i} \qquad \text{and} \qquad \overset{\circ}{E_{I}}:=E_{I}\backslash\cup_{j\in T\backslash I}E_{j}.
\]
In particular, when $I=\emptyset$ we have $E_{\emptyset}=Y$.\\

If $Z$ is a closed subscheme of $Y$, we denote the reduction mod $\cM_L$ of $Z$ by $\overline{Z}$ (see \cite{Shim}). We say that the resolution $(Y,h)$ of $f$ has \emph{good reduction modulo $\cM_L$} if $\overline{Y}$ and all $\overline{E_{i}}$ are smooth, $\cup_{i\in T}\overline{E_{i}}$ has normal crossings, and the schemes $\overline{E_{i}}$ and $\overline{E_{j}}$ have no common components whenever $i\neq j$. In addition, if $N_i\notin\cM_L$ for all $i\in T$, then  we say that $(Y,h)$ has {\em tame good reduction modulo $\cM_L$}.

Now we assume that $(Y,h)$ has tame good reduction modulo $\cM_L$. For $I\subset T$, one then verifies that $\overline{E}_{I}=\cap_{i\in I}\overline{E}_{i}$, and we put analogously $\overset{\circ}{\overline{E}}_{I}:=\overline{E}_{I}\backslash\cup_{j\notin I}\overline{E}_{j}$. Let $a$ be a closed point of $\overline{Y}$ and $T_{a}=\{i\in T \mid a\in \overline{E}_{i}\}$. In the local ring of $\overline{Y}$ at $a$ we can write
\[
\overline{f}\circ\overline{h}=\overline{u}\prod_{i\in T_{a}}\overline{g}_{i}^{N_{i}},
\]
where $\overline{u}$ is a unit, $(\overline{g}_{i})_{i\in T_{a}}$ is a part of a regular system of parameters and $N_{i}$ is as above. Each multiplicative character $\chi$ of conductor $1$ induces a multiplicative character $\chi:k_L^*\to\CC^*$. This induces the function $$\Omega_\chi:\overline{Y}(k_L)\to\CC : a \mapsto \chi(\overline{u}(a)). $$

\begin{thm}[\cite{Denef91}, Theorem 2.2 or \cite{DenefBourbaki}, Theorem 3.4]\label{goodreduction}
Let $L$ be a non-archimedean local field of characteristic $0$ and $f$  a non-constant polynomial in $L[x_1,\dots,x_n]$ with $\overline{f}:=f \mod \cM_L \neq 0$. Suppose that $D$ admits an embedded resolution $(Y,h)$ over $L$ with  tame good reduction modulo $\cM_L$. Let $\chi$ be a character on $\cO_{L}^*$ of order $d$, which is trivial on $1+\cM_L$. If $\Phi$ is residual, then we have
\[
Z(f,\chi,L,\Phi;s)=q_L^{-n}\sum_{\substack{I\subset T\\ \forall i\in I: d\mid N_{i}}}c_{I,\Phi,\chi}\prod_{i\in I}\frac{(q_L-1)q_L^{-N_{i}s-\nu_{i}}}{1-q_L^{-N_{i}s-\nu_{i}}},
\]
where
\[
c_{I,\Phi,\chi}=\underset{a\in\overset{\circ}{\overline{E}}_{I}(k_{L})}{\sum}\overline{\Phi}(\overline{h}(a))\Omega_{\chi}(a).
\]
\end{thm}

Let $f$ be a non-constant polynomial in $K[x_1,\dots,x_n]$ for some number field $K$ and fix an embedded resolution $(Y,h)$ of $f$ over $K$. Then, for all non-archimedean local fields $L$ with an embedding $K\to L$, we can consider $f$ as a polynomial over $L$ and $(Y,h)$ as an embedded resolution of $f$ over $L$. For all but finitely many  non-archimedean completions $L$ of $K$, that resolution has tame good reduction modulo $\cM_L$ (see \cite[Theorem 2.4]{Denef87}).
We also remark that, if $f$ has tame good reduction mod $\cM_L$, then it has tame good reduction mod $\cM_{L'}$ for all finite extension $L'$ of $L$.

\subsection{Motivic oscillation index}

We elaborate on the notion motivic oscillation index, introduced in \cite{CluckersIMRN} and further developed in \cite{CMN}.
We want to associate it to a polynomial $f$ over a number field and a reduced subscheme of affine space. It depends on poles of local zeta functions over all non-archimedean local fields $L\supset K$ for which $f$ is defined over $\cO_L$. To this end, we fix the following set-up.

Let $f\in\overline{\QQ}[x_1,\dots,x_n]$ be a non-constant polynomial and $Z$ a reduced subscheme of $\AA_{\overline{\QQ}}^n$.
Let $\cO$ be a finitely generated $\ZZ$-subalgebra of $\overline{\QQ}$ over which $f$ and $Z$ are defined.
With the latter we mean that $Z$ is induced by a subscheme of $\AA_{\cO}^n$, that we denote also by $Z$.
We say that a non-archimedean field $L$ of characteristic zero is {\em over $\cO$} if there exists an embedding $\cO\to \cO_L$.
For such $L$ we define the Schwartz-Bruhat function $\Phi_{L,Z}:=\11_{\{x\in\cO_L^n| \overline{x}\in Z(k_L)\}}$.

For a positive integer $M$, we denote by $\cL_M$
the set of all non-archimedean local fields of characteristic $0$ with residue field characteristic at least $M$, and
 by $\cL_{K,M}$
 the set of all non-archimedean local fields in $\cL_M$
endowed with an embedding  of a number field $K$.

\medskip
(1) We first consider the situation where $Z(\CC)\subset f^{-1}(0)$.
 The {\em motivic oscillation index of $f$ at $Z$ over a number field $K\supset\cO$} is given by
$$\moi_{K}(f,Z):=\liminf_{M\to +\infty, L\in\cL_{K,M}}\sigma_{f,L,\Phi_{L,Z}},$$
where $\sigma_{f,L,\Phi_{L,Z}}$ is as in Corollary  \ref{invariant sigma}.
More concretely, this means the following.
If $\moi_K(f,Z)<+\infty$, then there are infinitely many non-archimedean local fields $L$ over $\cO$ with an embedding $K\to L$,
such that $-\moi_{K}(f,Z)$ is the real part of a non-trivial pole of a zeta function $Z(f,\chi,L,\Phi_{L,Z};s)$, and there exists an integer $M$ such that, for all local fields $L\in\cL_{K,M}$, all non-trivial poles of $Z(f,\chi,L,\Phi_{L,Z};s)$ have real part at most $-\moi_{K}(f,Z)$.
Actually, $\moi_{K}(f,Z)=+\infty$ or $\moi_K(f,Z)$ is a positive rational number, equal to some $\frac{\nu_i}{N_i}, i\in T$, when fixing an embedded resolution of $\{f=0\}$. In particular, if $Z=\emptyset$, then we use the convention that $\moi_K(f,Z)=+\infty$ (we use this in (2) below).

\smallskip
(2) For general $Z$, we have to consider the (finitely many) critical values $c_i \in \overline{\QQ}$ of $f$. Put $\cO_i=\cO[c_i]$ and let $Z_i$ be the intersection of $Z\otimes_\cO\cO_i$ with $\{f=c_i\}$ in $\AA^n_{\cO_i}$. Then the motivic oscillation index of $f$ at $Z$ over a number field $K\supset\cO$ is $\moi_{K}(f,Z):=\min_i \moi_{K(c_i)}(f -c_i,Z_i)$.

\medskip
For $K \supset \cO$, take an embedded resolution $(Y,h)$ of $f^{-1}(0)$ over $K$. Then $(Y,h)$ induces an embedded resolution $(Y_L,h_L)$ over any  local field $L\in\cL_{K,1}$ over $\cO$. In particular, there exists an integer $M$ such that  $(Y_L,h_L)$ has tame good reduction modulo $\cM_L $ for all $L\in\cL_{K,M}$ over $\cO$.


Remark that we only talk about  Igusa local zeta functions over a local field  of characteristic $0$, since resolution of singularities is in general not known yet in positive characteristic, and hence a description in terms of an embedded resolution is not available for such zeta functions over a local field of positive characteristic. However, there exists an integer $M$ such that, for all non-archimedean local fields  $L$ of characteristic at least $M$ with an $\cO$-structure, the \lq induced $f$\rq\ in $L[x_1,\dots,x_n]$ has an embedded resolution with tame good reduction modulo $\cM_L$, and so the above claims can be extended to these fields. In particular, we have in this setting a good correspondence  between local fields in $\cL_{K,M}$ and their positive characteristic analogues with the same residue field. This is the content of transfer principles, on which we elaborate in Section \ref{section5}.


\section{Proof of Main theorem \ref{mainthm1}}\label{section3} Take polynomials $f,g$ and  $P,Q\in K^n$ as in the statement of Main theorem \ref{mainthm1}, satisfying thus
  $(f,P)\sim_K(g,Q)$. Since $(f,P)\sim_K (f(x+P),0)$ and the statement of Main theorem \ref{mainthm1} obviously holds for $(f,P)$ and $(f(x+P),0)$, we may and will suppose from now on that $P=Q=0$. The following lemma gives us (local) embedded resolutions of $f^{-1}(0)$ and $g^{-1}(0)$, which are compatible with the analytic isomorphism $\theta$.

\begin{lem}\label{anaiso} Let $K$ be a number field and $f,g$  two non-constant polynomials in $K[x_1,\dots,x_n]$ such that $f(0)=g(0)=0$. Suppose that there exist complex  neighbourhoods $U, V$ of $0\in\CC^n$, a biholomorphic function $\theta:U\to V$ satisfying $\theta(0)=0$, and an invertible holomorphic function $v:V\to \CC$ such that $f=(g.v)\circ\theta$. Then there exist embedded resolutions $(Y_1,h_1)$ of $(\AA_{\overline{K}}^n,f^{-1}(0))$ and $(Y_2,h_2)$ of $(\AA_{\overline{K}}^n,g^{-1}(0))$ at $0$  and a biholomorphic function $\Theta:\tilde{U}=h_1^{-1}(U)\to \tilde{V}=h_2^{-1}(V)$ such that $f\circ h_1|_{\tilde{U}}= (g.v)\circ h_2|_{\tilde{V}}\circ\Theta$.
\end{lem}

\begin{proof}  This is a consequence of functorial constructive embedded resolution, as in
\cite[Proposition 9.2]{Bravo-Encinas-Villamayor}, and in \cite{Temkin} for the complex analytic setting.
\end{proof}

We take $(Y_1,h_1), (Y_2,h_2), \Theta$ as in Lemma \ref{anaiso}. For each $a\in h_1^{-1}(0)$ there exists a Zariski neighbourhood $U_a$ of $a\in Y_1$ such that we can write
$f\circ h_1$ as $u_a\prod_{i\in T_a} z_i^{N_i}$ on $U_a$,
where $(z_i)_{i\in T_a}$ is a part of a regular system of parameters in $U_a$ and $u_a$ is an invertible regular function on $U_a$. Similarly, for each $b\in h_2^{-1}(0)$ there exists a Zariski neighbourhood $V_b$ of $b\in Y_2$ such that we can write
$g\circ h_2$ as $v_b\prod_{i\in T_b} x_i^{N_i}$ on $V_b$,
where $(x_i)_{i\in T_b}$ is a part of a regular system of parameters in $V_b$ and $v_b$ is an invertible regular function on $V_b$. Here we can and will assume that  $U_a$, $u_a$, all $z_i$, $V_b$, $v_b$ and all $x_i$ are defined over $\overline{K}$.

\begin{lem}\label{overC}With the notations above, there exists a $\CC$-isomorphism $\phi:h_1^{-1}(0)\to h_2^{-1}(0)$ of algebraic varieties, such that for each $a\in h_1^{-1}(0)$, for each choice of $U_a, u_a, z_i, V_{\phi(a)}, v_{\phi(a)},x_i$, for each smooth subvariety $W$ of the Zariski neighbourhood $W_a=(U_a\cap h_1^{-1}(0))\cap \phi^{-1}(V_{\phi(a)}\cap h_2^{-1}(0))$ of $a$ in $h_1^{-1}(0)$, there exists a regular function $r_{W}$ on $W$ such that $$u_a|_{W}= (v_{\phi(a)}|_{\phi(W)}\circ\phi) \cdot r_W^d ,$$ where $d=\gcd\{N_i \mid i\in T_a\}=\gcd\{N_i \mid i\in T_{\phi(a)}\}$.

\end{lem}
\begin{proof}
We take $\phi=\Theta|_{h_1^{-1}(0)}$, where $\Theta$ is the biholomorphic map from Lemma \ref{anaiso}. Hence, we know already that  $\phi$ is an analytic isomorphism from  $h_1^{-1}(0)$ onto $h_2^{-1}(0)$. But $h_1^{-1}(0)$ and $h_2^{-1}(0)$ can be viewed as algebraic closed subsets of suitable complex projective spaces, and then, by the GAGA theorem \cite{SerreGAGA}, we can suppose that $\phi$ is a $\CC$-isomorphism of two algebraic varieties. Lemma \ref{anaiso} shows further that
$$(u_a \cdot\prod_{i\in T_a} z_i^{N_i})|_{\tilde{W}_a}=((v\circ h_2)\cdot v_{\phi(a)}\cdot \prod_{i\in T_{\phi(a)}} x_i^{N_i})|_{\Theta(\tilde{W}_a)}\circ\Theta$$
on the open subset $\tilde{W}_a=\tilde{U}\cap U_a\cap \Theta^{-1}(V_{\phi(a)}\cap\tilde{V})$. So we can identify  $T_a$ with $T_{\phi(a)}$ and suppose that $x_i|_{\Theta(\tilde{W}_a)}\circ\Theta= z_i|_{\tilde{W}_a}\cdot t_i$ for an invertible holomorphic function $t_i$ on $\tilde{W}_a$. Denoting $\tilde{t}=\prod_{i\in T_a}t_i^{N_i/d}$,  we have
$$u_a|_{\tilde{W}_a}=\tilde{t}^d \cdot ((v\circ h_2) \cdot v_{\phi(a)})|_{\Theta(\tilde{W}_a)}\circ\Theta$$
on $\tilde{W}_a$. So for each smooth subvariety $W$ of $W_a$ we have
$$u_a|_{W}=t_W^d \cdot (v_{\phi(a)}|_{\phi(W)}\circ\phi),$$
where $t_W$ is the non-vanishing holomorphic function $\alpha \tilde{t}|_{W}$ for some $d^{th}$ root $\alpha$ of $v(0)$. Lemma \ref{holo-regular} below shows that $t_W$ is a regular function on $W$, which finishes the proof.
\end{proof}

\begin{lem}\label{holo-regular} Let $X$ be a smooth quasi-projective variety over $\CC$ and $w$ an invertible regular function on $X$. Suppose that there exists a holomorphic function $t:X\to \CC$ such that $w=t^d$ for some positive integer $d$. Then $t$ is a regular function on $X$.
\end{lem}

\begin{proof}
Suppose that there is no regular function $r$ on $X$ such that $w=r^d$. We may then also assume that there is no regular function $w'$ on $X$ such that $w=(w’)^{d/d'}$ for some non-trivial divisor $d'$ of $d$.  (Indeed, otherwise we can reduce to this case by considering $w'$ and  $d'$ instead of $w$ and $d$.)

Since $X$ is smooth, $\cO_X(X)$ is an integrally closed domain, and hence the Kummer covering $Y$ of $X$ given by $Y=\{(x,y)\in X\times\CC^* \mid w(x)=y^d\}$ is irreducible for the Zariski topology, and consequently connected for the complex topology. Since $t$ is a holomorphic function, we have that  $Y_i=\{(x,y)\in Y \mid y=\alpha_it(x)\}, 1\leq i\leq d$, are closed subsets of $Y$ for the complex topology, where $\alpha_1,\dots,\alpha_d$ are the $d^{th}$ roots of $1$. But $Y_i\cap Y_j=\emptyset$ since $t(x)\neq 0$ for all $x\in X$. We obtained a contradiction with the connectedness of $Y$ for the complex topology. So there exists a regular function $r:X\to\CC$ such that $r^d=w$. Then the sets $X_i=\{x\in X \mid t(x)=\alpha_i r(x)\}$ are disjoint closed subsets for the complex topology of $X$, since $t(x)\neq 0$ for all $x\in X$. On the other hand, we have that $X$ is connected for the complex topology, so there exists exactly one index $i$ such that $X_i=X$ and $X_j=\emptyset$ if $j\neq i$. So $t=\alpha_i r$ is a regular function on $X$.
\end{proof}

Our final task is to show that the isomorphism $\phi$ in Lemma \ref{overC} can  be defined over $\overline{K}$. This will follow from the fact that the existence of $\phi$ can be described as a finite set of polynomial conditions. To this end we fix some notation and terminology.

Let $h_1^{-1}(0)=\bigsqcup_i W_i$ be the canonical stratification of $h^{-1}(0)$, i.e., $W_{i+1}$ is the set of smooth points of $\operatorname{Sing}(\overline{W_i})$, and let $W_i=\bigsqcup_{j}W_{ij}$ be the decomposition of $W_i$ into irreducible components. Then all $W_{ij}$ are smooth and defined over $\overline{K}$. Since $\phi: h_1^{-1}(0)\to h_2^{-1}(0)$ is an isomorphism, the  varieties $\phi(W_{ij})$ give us the canonical stratification of $h_2^{-1}(0)$, and hence all $\phi(W_{ij})$ are also smooth and can be defined over $\overline{K}$.

We capture a key property of $\phi$ in the following definition. We say that a map $\varphi:h_1^{-1}(0)\to h_2^{-1}(0)$ {\em preserves multiplicities} if $|T_a|=|T_{\varphi(a)}|$ and we have an equality between the two tuples  $(N_i)_{i\in T_a}$ and $(N_i)_{i\in T_{\varphi(a)}}$ (with a suitable order on $T_a$ and $T_{\varphi(a)}$). In particular, if for each irreducible component $E$ of $(f\circ h_1)^{-1}(0)$ with $E\cap h_1^{-1}(0)\neq\emptyset$ and the unique divisor $F$ of $(g\circ h_2)^{-1}(0)$ such that $\Theta(E\cap \tilde{U})\subset F$, we have $\varphi(E\cap h_1^{-1}(0))\subset F\cap h_2^{-1}(0)$, then $\varphi$ preserves multiplicities.

\begin{lem}\label{overK} Using the notation above, there exists a $\overline{K}$-isomorphism $\varphi:h_1^{-1}(0)\to h_2^{-1}(0)$, such that $\varphi$ preserves multiplicities  and, for all $ij$ and all $a\in W_{ij}$, there exist a Zariski neighbourhood $W_a$ (resp. $V_{\varphi(a)}$) of $a$ in $Y_1$ (resp. of $\varphi(a)$ in $Y_2$) and functions $u_a, (z_i)_{i\in T_a}$ (resp. $v_{\varphi(a)}, (x_i)_{i\in T_{\varphi(a)}}$) as before, all defined over $\overline{K}$, and  a $\overline{K}$-regular function  $r_{W_{ija}}$ on  $W_{ija}:=W_{ij}\cap W_a\cap \varphi^{-1} (V_{\varphi(a)}\cap h_2^{-1}(0))$ such that
$$u_a|_{W_{ija}}=(v_{\varphi(a)}|_{\varphi(W_{ija})}\circ\varphi) \cdot r_{W_{ija}}^{d_a},$$
 where $d_a=\gcd\{N_i \mid i\in T_a\}=\gcd\{N_i \mid i\in T_{\varphi(a)}\}$.
\end{lem}

\begin{proof} We use the notation of Lemma \ref{overC}. We can fix finitely many points $a_1,\dots,a_{m}\in h_1^{-1}(0)(\overline{K})$ and $b_{1},\dots,b_{m'}\in h_2^{-1}(0)(\overline{K})$ with, for $ 1\leq \ell\leq m$ and $1\leq \ell'\leq m'$,  parameters
$$U_{a_{\ell}}, u_{a_{\ell}}, (z_{\ell j})_{j\in T_{a_{\ell}}} \quad\text{and}\quad  V_{b_{\ell'}},  v_{b_{\ell'}}, (x_{\ell'j})_{j\in T_{b_{\ell'}}}  ,$$
all defined over $\overline{K}$, such that $(U_{a_{\ell}})_{1\leq \ell\leq m}$ and $(V_{b_{\ell'}})_{1\leq \ell'\leq m'}$ are affine coverings  of $h_1^{-1}(0)$ and $h_2^{-1}(0)$, respectively,  and  for each $a\in h_1^{-1}(0)$ and $b\in h_2^{-1}(0)$, there exist $\ell$ with $a\in U_{a_{\ell}}$ and $T_a=T_{a_{\ell}}$, and $\ell'$ with $ b\in V_{b_{\ell'}}$ and $ T_b=T_{b_{\ell'}}$, respectively.
Then for $\ell,\ell'$ such that $(N_e)_{e\in T_{a_{\ell}}}=(N_e)_{e\in T_{b_{\ell'}}}$ and $ij$ such that $W_{ij\ell \ell'}:=W_{ij}\cap U_{a_{\ell}}\cap \phi^{-1}(V_{b_{\ell'}}\cap h_2^{-1}(0))\neq \emptyset$, there exists $r_{W_{ij\ell\ell'}}$ satisfying Lemma \ref{overC} with respect to $u_{a_{\ell}}, v_{b_{\ell'}}, \phi$ and $d=d_{a_{\ell}}=d_{b_{\ell'}}$. We may suppose that $\phi$, $\phi^{-1}$ and all $ r_{W_{ij\ell\ell'}}$ are  defined by polynomials of degree at most $D$.

 We remark that $a\mapsto T_a$ and $a\mapsto (N_i)_{i\in T_a}$ are $A$-definable in the language of rings, where
$$A=\{f, g, h_1, h_2, \big(U_{a_{\ell}},u_{a_{\ell}},(z_{\ell j})_{j\in T_{a_{\ell}}}\big)_{1\leq \ell\leq m}, \big(  V_{b_{\ell'}},v_{b_{\ell'}},(x_{\ell' j})_{j\in T_{b_{\ell'}}}\big)_{ 1\leq \ell'\leq m'}  \} .$$
Consequently, the following first-order formula  in the language of rings with parameters $A$ and the $W_{ij}$, defined over $\overline{K}$, will hold over $\CC$.
{\em There exists  an isomorphism $\varphi$ between $h_1^{-1}(0)$ and $h_2^{-1}(0)$ such that
\begin{itemize}
\item[(i)] $\varphi$ preserves multiplicities,
\item[(ii)] for each tuple $(i,j,\ell,\ell')$ such that $(N_e)_{e\in T_{a_{\ell}}}=(N_e)_{e\in T_{b_{\ell'}}}$ and $W_{\varphi,ij\ell\ell'}:=W_{ij}\cap (U_{a_{\ell}}\cap h_1^{-1}(0))\cap \varphi^{-1}(V_{b_{\ell'}}\cap h_2^{-1}(0))\neq \emptyset$, there exists a regular function $s_{W_{\varphi,ij\ell\ell'}}$ on $W_{\varphi,ij\ell\ell'}$ such that
$$u_{a_{\ell}}|_{W_{\varphi,ij\ell\ell'}}=(v_{b_{\ell'}}|_{\varphi(W_{\varphi,ij\ell\ell'})}\circ\varphi) \cdot s_{W_{\varphi,ij\ell\ell'}}^{d},$$
 where $d=\gcd\{N_i|i\in T_{a_{\ell}}\}=\gcd\{N_i|i\in T_{b_{\ell'}}\}$,
\item[(iii)] $\varphi,\varphi^{-1}$ and all $ s_{W_{\varphi,ij\ell\ell'}}$ are defined by polynomials of degree at most $D$.
\end{itemize}}
By quantifier elimination for the theory of algebraically closed fields of characteristic $0$ (or by an application of the Nullstellensatz), it also holds over $\overline{K}$, i.e., there exists a solution $(\varphi,s_{W_{\varphi,ij\ell\ell'}})$ defined over $\overline{K}$. Now for each pair $ij$ and each $a\in W_{ij}$, there exist $\ell,\ell'$ such that $a\in U_{a_{\ell}}, \varphi(a)\in V_{b_{\ell'}}$ and $T_a=T_{a_{\ell}}, T_{\varphi(a)}=T_{b_{\ell'}}$. Since $\varphi$ preserves multiplicities, we have that $(N_e)_{e\in T_{a_{\ell}}}=(N_e)_{e\in T_{b_{\ell'}}}$, and hence our claim follows by (ii) with the remark that $W_{\varphi,ij\ell\ell'}\neq \emptyset$.
\end{proof}

\begin{proof}[Proof of Main theorem \ref{mainthm1}]
We fix the following data: $h_1,h_2$ from Lemma \ref{anaiso}, and $\varphi$, $\varphi^{-1}$, a finite covering $\{U_a \mid a\in \cA\}$ of a Zariski neighbourhood of $h_1^{-1}(0)$ in $Y_1$ (where  $\cA\subset h_1^{-1}(0)$), a finite covering $\{V_{\varphi(a)}\mid a\in \cA\}$ of a Zariski neighbourhood of $h_2^{-1}(0)$ in $Y_2$, regular functions $(u_a, (z_{i})_{i\in T_a} , v_{\varphi(a)}, (x_{i})_{i\in T_{\varphi(a)}} )_{a\in\cA}$, and regular functions $r_{W_{ija}}$ for all $ij$ and $a$ as in Lemma \ref{overK}, such that everything is defined over $\overline{K}$. Moreover, by the proof of Lemma \ref{overK} we can assume that, for each $\tilde a\in h_1^{-1}(0)$, there exists $a\in\cA$ such that $\tilde a\in (U_{a}\cap h_1^{-1}(0))\cap \varphi^{-1}(V_{\varphi(a)}\cap h_2^{-1}(0))$ and $T_a=T_{\tilde a}$. Let $K'$ be the smallest number field over which these (finitely many) data can be defined. There exists a positive integer $M$ such that, for each local field $L\in\cL_{K',M}$, the data can be defined over $\cO_L$  and $h_1,h_2$ have tame good reduction modulo  $\cM_L$.  The theorem now follows from Lemmas \ref{zerofunction}, \ref{goodreduction}, \ref{anaiso}, \ref{overK} and the proof of Theorem 2.2 in \cite{Denef91}.
 \end{proof}

 \begin{proof}[Proof of Proposition \ref{conj2}]
 Given that $(f,P)$ and $(g,Q)$ are similar over $K$,  Theorem \ref{mainthm1} implies that there exists a finite extension $K'$ of $K$ such that $\moi_{K'}(f,P)=\moi_{K'}(g,Q)$. Assuming that Conjecture \ref{Conj1} holds, we obtain $\moi_K(f,P)=\moi_{K'}(f,P)=\moi_{K'}(g,Q)=\moi_K(g,Q)$.
\end{proof}

\section{Positive characteristic analogues and cohomological description}\label{section4}

In this section, we study analogous exponential sums over local fields of positive characteristic.
We need here the notion of jet schemes, associated to a scheme over a finite field. Since we will use this notion later also for schemes over $\CC$, we  introduce the relevant notation in general.

\subsection{Jet schemes}


Let $k_0$ be any field, $f\in k_0[x_1,\dots,x_n]$ a non-constant polynomial and $Z$ a reduced subscheme of $\AA_{k_0}^n$.

For a non-negative integer $r$, the $r$th jet scheme of a scheme $W$ (of finite type) over $k_0$ is the scheme $W_r$ over $k_0$, determined by the property $W_r(R)=\operatorname{Hom}_{k_0}(\Spec R[t]/(t^{r+1}),W_r)$ for any $k_0$-algebra $R$. Note that $W_0=W$.

We denote here $A=\AA_{k_0}^n$. Its $r$th jet scheme $A_r$ can be identified with the affine space $\AA_{k_0}^{n(r+1)}$.
For $r\geq r'$, we denote by $\pi_{rr'} : A_r\to A_r'$ the canonical projection induced by the homomorphism $R[t]/(t^{r+1})\to R[t]/(t^{r'+1})$.

The reduced $r$th jet scheme $X_{r,red}$ of $X=\Spec k_0[x_1,\dots,x_n]/(f)$ can be described as follows. We write
\begin{align*}
   & f(x_{01}+x_{11}t+\dots+x_{r1}t^r,\dots, x_{0n}+x_{1n}t+\dots+x_{rn}t^r) \\
 = & f_0+ f_1t+\dots+f_rt^{r} \mod t^{r+1},
\end{align*}
where $x^{(r)}=(x_{ij})_{0\leq i\leq r, 1\leq j\leq n}$ are coordinates on $A_r=\AA_{k_0}^{n(r+1)}$, given by the identification $A_r(k)=k^{(r+1)n} \cong (k[t]/(t^{r+1}))^n$ for any $k_0$-field $k$.
Hence, for all $\ell=0,\dots, r$, we have that $f_\ell$ is a polynomial in $x^{(\ell)}$. Then $X_{r,red}$ can be identified with the algebraic subset of $A_r$ given by the equations $f_0=\dots=f_r=0$.


We consider various lifts of $Z\subset A=\AA_{k_0}^n$ to $A_r=\AA_{k_0}^{n(r+1)}$. First, we denote $Z^{(r)}=\pi_{r0}^{-1}(Z)$. For any $k_0$-field $k$, we can identify $Z^{(r)}(k)$  with $(Z(k)+ tk[[t]]^n)/(t^{r+1}k[[t]]^n)$.
Further, for each $1\leq i \leq  r$, we denote by $Z_{r,i}$ the algebraic subset of $Z^{(r)}$ given by the conditions  $f_1=\dots=f_i=0$, and we put $Z_{r,0}=Z^{(r)}$.  Hence
$$Z_{r,r}\subset Z_{r,r-1} \subset Z_{r,r-2} \subset \dots \subset Z_{r,2} \subset Z_{r,1} \subset Z_{r,0}=Z^{(r)}.$$

We will write $Z_f^{(r)}$ and $Z_{f,r,i}$ instead of $Z^{(r)}$ and $Z_{r,i}$ if we want to emphasize $f$.

\subsection{Exponential sums in positive characteristic}

Let $p$ be a prime number and $\FF_q$ a finite field of characteristic $p$. Let
$$\Psi_p: \FF_p\to\CC^{*} : (x \mod p)\mapsto \exp\left(\frac{2\pi ix}{p}\right)$$
denote the standard additive character of $\FF_p$.  For any finite extension $k$ of $\FF_p$, we consider the standard additive character
$$\psi:k((t))\to \CC^{*} : \sum_{i\geq N}b_it^i \mapsto \Psi_p(\Tr_{k/\FF_p}(b_{-1})) ,$$
where $\Tr_{k/\FF_p}:k\to\FF_p$ is the trace map. If $z=\sum_{i\geq -m}a_it^i\in k((t))$ with $a_{-m}\neq 0$, we consider the additive character
$$\psi_z:k((t))\to\CC^* : x \mapsto \psi(xz).$$
Then $\psi_z(\sum_{i\geq N}b_it^i)= \Psi_p(\Tr_{k/\FF_p}(\sum_{i}b_ia_{-1-i}))$. In particular, if $x=\sum_{i\geq 0}b_it^i\in k[[t]]$, then
$$\psi_z(x)=\Psi_p\bigl(\Tr_{k/\FF_p}(\sum_{0\leq i\leq m-1}b_ia_{-1-i})\bigr).$$

\medskip
\begin{defn}\label{finitexp}
Let $f\in \FF_q[x_1,\dots,x_n]$ be a non-constant polynomial and $Z$ a reduced subscheme of $\AA_{\FF_q}^n$.
Fix a positive integer $m>1$, a finite extension $k$ of $\FF_q$, and $z=\sum_{i\geq -m}a_it^i\in k((t))$ with $a_{-m}\neq 0$. Using the notation above, the exponential sum associated to $f$, $z$ and $Z$ is
\begin{align*}
E_{f,k((t)),Z,\psi_z}:=&\int_{Z(k)+tk[[t]]^n}\psi_z(f(x))|dx| \\
=&\#(k)^{-mn}\sum_{Z^{(m-1)}(k)}\Psi_p \Bigl( \Tr_{k/\FF_p}\bigl(\sum_{i=0}^{m-1}a_{-i-1}f_{i}(x^{(i)})\bigr)\Bigr).
\end{align*}
\end{defn}
It turns out that, when $p$ is big enough with respect to $m$, the expression in Definition \ref{finitexp} can be simplified a lot. More precisely, by reparameterizing the jets, we can arrange that only $f_0$ and $f_{m-1}$ occur in the last sum. The following lemma contains the main argument.

\begin{lem}\label{changechar}Let $k$ be a finite extension of $\FF_p$. Suppose that $p>m>2$ and let $z=\sum_{i\geq -m}a_it^i\in k((t))$ with $a_{-m}\neq 0$. Then there exists $t_1\in k((t))$ with $\ord_t(t_1)=1$ and $\psi_z(at_1^{\ell})=1$ for all $1\leq \ell\leq m-2$ and $a\in k$.
\end{lem}

\begin{proof}
We will find $t_1$ of the form $t_1=t+\alpha_2t^2+\dots+\alpha_{m-1}t^{m-1}$ (with the $\alpha_i \in k$). We remark that
\begin{align*}
t_1^{\ell}=&t^{\ell}(1+\alpha_2t+\dots+\alpha_{m-1}t^{m-2})^{\ell}\\
=&t^{\ell}\Bigl(1+u_{1,\ell}(\alpha_2)t+u_{2,\ell}(\alpha_2,\alpha_3)t^2+\dots+u_{m-\ell-2,\ell}(\alpha_2,\dots,\alpha_{m-\ell-1})t^{m-\ell-2}\\
&+(u_{m-\ell-1,\ell}(\alpha_2,\dots,\alpha_{m-\ell-1})+\ell\alpha_{m-\ell})t^{m-\ell-1}+u_\ell t^{m-\ell}\Bigr),
\end{align*}
for all $1\leq \ell\leq m-2$, where $u_{1,\ell},\dots,u_{m-\ell-1,\ell}$ are polynomials and $u_{\ell}\in k[[t]]$. By definition of $\psi_z$, we need to find $\alpha_2,\dots,\alpha_{m-1}\in k$ such that
\begin{equation}\label{formula changechar}
a_{-1-\ell}+a_{-2-\ell}u_{1,\ell}(\alpha_2)+\dots+a_{-m}(u_{m-\ell-1,\ell}(\alpha_2,\dots,\alpha_{m-\ell-1})+\ell\alpha_{m-\ell})=0
\end{equation}
for all $1\leq \ell\leq m-2$. Since $a_{-m}\neq 0$ and  $\ell<m<p$, we can find $\alpha_2,\dots,\alpha_{m-1}$ by induction as follows:
the cases $\ell=m-2,m-3,\dots,1$ of equation (\ref{formula changechar}) determine $\alpha_2$, then $\alpha_3$ in terms of $\alpha_2$, \dots, and finally $\alpha_{m-1}$ in terms of all previous $\alpha_i$, respectively.
\end{proof}

\begin{prop}\label{simplified expsum}
Let $f\in \FF_q[x_1,\dots,x_n]$ be a non-constant polynomial and $Z$ a reduced subscheme of $\AA_{\FF_q}^n$.   Let $m\in \ZZ_{\geq 2}$ and suppose that $p>m$ if $m>2$. Take   a finite extension $k$ of $\FF_q$ and $z\in k((t))$ with $\ord_t(z)=-m$. Then there exists $b_0\in k$ and $b_{m-1}\in k^{*}$ such that
\begin{align*}\label{reduceformu}
E_{f,k((t)),Z,\psi_z}=&\int_{Z(k)+tk[[t]]^n}\psi_z(f(x))|dx|\\
=&\#(k)^{-mn}\sum_{Z^{(m-1)}(k)}\Psi_p\bigl(\Tr_{k/\FF_p}(b_0f_0(x^{(0)})+b_{m-1}f_{m-1}(x^{(m-1)}))\bigr).
\end{align*}
\end{prop}

\begin{proof}
If $m=2$, there is nothing to prove. We thus assume that $p>m>2$. We take $t_1$ as in Lemma \ref{changechar} with respect to $z$. We consider now the standard additive character $\tilde{\psi}:k((t_1))\to \CC^*$  of $k((t_1))$ as above.
Since $k((t))=k((t_1))$, there exists $z'$ with $\ord_{t_1}(z')=-m$ and $\psi_z=\tilde{\psi}_{z'}$, where $\tilde{\psi}_{z'}(x)=\tilde{\psi}(z'x)$. Writing $z'=\sum_{i\geq -m}a_i't_1^i$,  Lemma \ref{changechar} implies that we must have $a'_i=0$ for all $-m+1\leq i\leq -2$.

Since $f\in\FF_q[x_1,\dots,x_n]$, we can write
\begin{align*}
   & f(x_{01}+x_{11}t_1+\dots+x_{m1}t_1^m,\dots, x_{0n}+x_{1n}t_1+\dots+x_{mn}t_1^m) \\
 = & f_0+ f_1t_1+\dots+f_mt_1^{m} \mod t_1^{m+1},
\end{align*}
 for all $m\geq 0$, where $f_0, f_1, \dots,f_m$ are the same polynomials as before. Hence we have
\begin{align*}
&\int_{Z(k)+tk[[t]]^n}\psi_z(f(x))|dx|\\
=&\int_{Z(k)+t_1k[[t_1]]^n}\tilde{\psi}_{z'}(f(x))|dx|\\
=&\#(k)^{-mn}\sum_{Z^{(m-1)}(k)}\Psi_p \bigl( \Tr_{k/\FF_p}(a'_{-1}f_{0}(x^{(0)})+  a'_{-m}f_{m-1}(x^{(m-1)}))\bigr).
\end{align*}
\end{proof}

\subsection{Cohomological description of exponential sums}

Using the simplification above, we can now use Deligne's theory on exponential sums over finite fields to study the exponential sums $E_{f,k((t)),Z,\psi_z}$. This requires the notion of Artin-Schreier sheaves, which we first recall.

\smallskip
Let $\FF_q$ be a finite field of characteristic $p>0$. We consider the Artin-Schreier covering $\cX_0$ of $\AA_{\FF_q}^1$, given by the equation $T^p-T=x$ in the affine plane over $\FF_q$. Hence $\cX_0=\Spec \FF_q[x,T]/(T^p-T-x)$ with Galois group  $\FF_p$ (corresponding to the maps $T\mapsto T+a$ for $a\in \FF_p$).

We fix a prime number  $\ell\neq p$ and a non-trivial additive character $\Psi:\FF_p\to\overline{\QQ}_{\ell}^*$ of $\FF_p$, induced by $\Psi_p$ and an isomorphism $\iota$ between $\overline{\QQ}_\ell$ and $\CC$. We consider the $\ell$-adic local system $\cL_{\Psi}$ on $\AA^1_{\FF_q}$, given by pushing out the Galois covering by the character $\Psi^{-1}:\FF_p\to \overline{\QQ}_{\ell}^*$.

\begin{notation}\label{notation for finite fields}
Let $f\in \FF_q[x_1,\dots,x_n]$ be a non-constant polynomial, $Z$ a reduced subscheme of $\AA_{\FF_q}^n$, and $m\in \ZZ_{\geq 2}$. To a finite extension $k$ of $\FF_q$ and  $(b_0,b_{m-1})\in k\times k^*$,
we associate the polynomial $b_0f_0 + b_{m-1}f_{m-1}$ as in Proposition \ref{simplified expsum}, considered as
$k$-morphism $b_0f_0 + b_{m-1}f_{m-1}:Z_k^{(m-1)}\to \AA_{k}^1$.
\end{notation}

\begin{defn} With Notation \ref{notation for finite fields}, the {\em Artin-Schreier sheaf} $\cL_{f,b,\Psi}$ associated with $b_0f_0 + b_{m-1}f_{m-1}$ and $\Psi$ is the $\ell$-adic local system $\cL_{f,b,\Psi}:=(b_0f_0 + b_{m-1}f_{m-1})^{*}(\cL_{\Psi})$ on $Z_k^{(m-1)}$.
 \end{defn}

If $k=\FF_{q^r}$ and $x^{(m-1)}\in Z_k^{(m-1)}(k)$, then the trace of the geometric Frobenius element $\Frob_{q}^r$ in $\Gal(\overline{\FF}_p,k)$, acting on the fiber of $\cL_{f,b,\Psi}$ at a geometric point $\overline{x}$ over $x^{(m-1)}$, is $\Psi \bigl(\Tr_{k/\FF_p}((b_0f_0 + b_{m-1}f_{m-1})(x^{(m-1)}))\bigr)$. As usual, if $p>m$, then Proposition \ref{simplified expsum} and  the Grothendieck trace formula yield
\begin{align*}
E_{f,k((t)),Z,\psi_z}=&\iota\Bigl(\sum_{x^{(m-1)}\in Z^{(m-1)}(k)}\Psi\bigl(\Tr_{k/\FF_p}((b_0f_0 + b_{m-1}f_{m-1})(x^{(m-1)})) \bigr)\Bigr) \\
=&\iota\Bigl(\sum_{i=0}^{2(\dim(Z)+n(m-1))}(-1)^i\Tr\bigl(\Frob_{q}^r, \H_c^i(Z^{(m-1)}\otimes_{\FF_q}\overline{\FF}_p, \cL_{f,b,\Psi})\bigr)\Bigr).
\end{align*}
Note here that $\dim(Z)+ n(m-1) = \dim(Z^{(m-1)})$.

\medskip
We recall that an element $\alpha\in\overline{\QQ}_{\ell}$ is said to be {\em pure of weight $j\in\ZZ$ related to $\FF_q$}, if $|\iota(\alpha)|=q^{j/2}$ for every embedding $\iota$ of the subfield $\QQ(\alpha)\subset \overline{\QQ}_{\ell}$ into $\CC$. Deligne showed in \cite{DeligneWeil2} that every eigenvalue of $\Frob_q$ on the finite dimensional $\overline{\QQ}_{\ell}$-vector space $\H_c^i(Z^{(m-1)}\otimes_{\FF_q}\overline{\FF}_p, \cL_{f,b,\Psi})$ is pure of some weight $j\in [0,i]$ related to $\FF_q$.

\begin{defn}Let as above $f$ be a polynomial in $\FF_q[x_1,\dots,x_n]$, $Z$  a reduced subscheme of $\AA_{\FF_q}^n$ and $k$ a finite extension of $\FF_q$.
The {\em highest weight of $f$ on $Z$} related to $m\in \ZZ_{\geq 2}$ and $b=(b_0,b_{m-1})\in k\times k^*$, denoted by $w(f,b,Z^{(m-1)})=w(b_0f_0 + b_{m-1}f_{m-1},Z_k^{(m-1)})$ is the highest weight among  the weights related to $\FF_q$ of all eigenvalues of $\Frob_q$ on all cohomology groups $\H_c^i(Z^{(m-1)}\otimes_{\FF_q}\overline{\FF}_p, \cL_{f,b,\Psi})$, with the convention that $w(f,b,Z^{(m-1)})=-\infty$ if all those cohomology groups vanish.
\end{defn}

It is clear that, if $p>m$ and we know the value of $w(f,b,Z^{(m-1)})$, then we know an upper bound for the exponential sums $E_{f,k((t)),Z,\psi_z}$. However, for an arbitrary  variety $\cY$ over $\FF_q$ and an arbitrary regular function $g:\cY\to \AA_{\FF_q}^1$, we do not know yet that such a weight $w(g,\cY)$ gives us the best exponent in the estimate of the exponential sums over $\FF_q$ associated with  ($g,\cY$). And even if $w(g,\cY)$ gives us the best exponent over $\FF_q$, we  do not know what happens if we consider a finite extension of $\FF_q$.
In a special case, Bombieri and Katz answered some questions about this problem by using the subspace theorem (see \cite{Bombieri-Katz}), namely, when the exponential sum is {\em pure}, i.e., $\H^i_c(\cY\otimes_{\FF_q}\overline{\FF}_p, \cL_{g,\Psi})=0$ if $i\neq\dim(\cY)$ and all the eigenvalues of $\Frob_q$ on $\H^{\dim(\cY)}_c(\cY\otimes_{\FF_q}\overline{\FF}_p, \cL_{g,\Psi})$ are of weight $\dim(\cY)$.

\subsection{Vanishing of cohomology}

Our exponential sum
 is rarely pure, so we need to do more if we want to understand the behaviour of the exponential sum under extension  of the base field.
Firstly, we can easily show that $\H_c^i(Z^{(m-1)}\otimes_{\FF_q}\overline{\FF}_p, \cL_{f,b,\Psi})$ vanishes  for small enough $i$.

\begin{prop}Suppose that there is a stratification of $Z$ with all strata affine and smooth over $\FF_q$. Then
$$\H_c^i(Z^{(m-1)}\otimes_{\FF_q}\overline{\FF}_p, \cL_{f,b,\Psi})=0  \qquad\text{for all } i< n(m-1).                   $$
More precisely, let $s$ be the smallest occurring dimension of these strata. Then
$$\H_c^i(Z^{(m-1)}\otimes_{\FF_q}\overline{\FF}_p, \cL_{f,b,\Psi})=0  \qquad\text{for all } i<s + n(m-1).                                                 $$
In particular, if $Z$ is affine and smooth, then
$$\H_c^i(Z^{(m-1)}\otimes_{\FF_q}\overline{\FF}_p, \cL_{f,b,\Psi})=0     \qquad\text{for all } i< \dim(Z) + n(m-1).                                          $$
\end{prop}

\begin{proof}
If $U$ is a stratum of dimension $s'\geq s$, then $U\times\AA_{\FF_q}^{n(m-1)}$ is smooth and affine of dimension $s'+n(m-1)$. So, by Poincar\'e duality, we obtain that
$$\H_c^i(U\otimes_{\FF_q}\overline{\FF}_p, \cL_{f,b,\Psi})=0$$
for all $0\leq i< s'+n(m-1)$. The claim follows from the spectral sequence in \cite[Sommes trig. 2.5*]{Deligneetalecoho}.
\end{proof}

Now we show a fibration property for $b_0f_0+b_{m-1}f_{m-1}$ that helps us to compute $\H_c^i(Z^{(m-1)}\otimes_{\FF_q}\overline{\FF}_p, \cL_{f,b,\Psi})$ for large enough $i$.

\begin{lem}\label{fibration} We use Notation \ref{notation for finite fields}. Suppose that $p>m >i_0+2 \geq 2$. Then, for $(b_0,b_{m-1})\in k\times k^*$, the restriction of $b_0f_0+b_{m-1} f_{m-1}$ to $Z_{k,m-1,i_0}\setminus Z_{k,m-1,i_0+1}$ is a  trivial fibration, i.e., there exists a scheme $V_{0}$ and an $k$-isomorphism $\theta:Z_{k,m-1,i_0}\setminus Z_{k,m-1,i_0+1} \to V_{0}\times\AA_{k}^1$ such that $\pi\circ\theta=(b_0f_0+b_{m-1}f_{m-1})|_{Z_{k,m-1,i_0}\setminus Z_{k,m-1,i_0+1}}$, where $\pi$ is the projection from $V_{0}\times\AA_{k}^1$ to $\AA_{k}^1$.
\end{lem}

\begin{proof}
To simplify notation, we write $V=Z_{k,m-1,i_0}\setminus Z_{k,m-1,i_0+1}$  and we take $V_{0}=(b_0f_0+b_{m-1}f_{m-1}|_{V})^{-1}(0)$. We will construct a $k$-isomorphism $\theta:V\to V_{0}\times\AA_{k}^1$, using the reparameterization idea as in the proof of Lemma \ref{changechar}. More precisely, we consider
$t_1= \sum_{i\geq 1}\alpha_it^i$ with all $\alpha_i\in k$ and moreover $\alpha_1=1$.
Let $x^{(m-1)}=(x_{ij})_{0\leq i\leq m-1, 1\leq j\leq n}\in V$, where we use the coordinate system on $A_{m-1}$ with respect to the parameter $t$.
We also consider the  $(m-1)$-jet
$$  u:=(x_{01}+x_{11}t_1+\dots+x_{(m-1)1}t_1^{m-1},\dots, x_{0n}+x_{1n}t_1+\dots+x_{(m-1)n}t_1^{m-1})$$
in $k[[t]]/(t^m)$, that can alternatively be written in the form
$$  u=(y_{01}+y_{11}t+\dots+y_{(m-1)1}t^{m-1},\dots, y_{0n}+y_{1n}t+\dots+y_{(m-1)n}t^{m-1}).$$
Then $y_{ij}=g_{i}(x_{0j},\dots,x_{ij},\alpha_1,\dots,\alpha_i)$ for all $ij$, where $g_i, 0\leq i\leq m-1,$ are  polynomials with coefficients in $k$. In particular, we have $x^{(0)}=y^{(0)}$.  We have also that
$$f(u)=f_0(x^{(0)})+f_{i_0+1}(x^{(i_0+1)})t_1^{i_0+1}+\dots+f_{m-1}(x^{(m-1)})t_1^{m-1} \mod t^{m}$$
and
$$f(u)=f_0(y^{(0)})+f_{i_0+1}(y^{(i_0+1)})t^{i_0+1}+\dots+f_{m-1}(y^{(m-1)})t^{m-1} \mod t^{m},$$
with $f_{i_0+1}(x^{(i_0+1)})=f_{i_0+1}(y^{(i_0+1)})\neq 0$ since $x^{(m-1)}\in V$ and $t_1= t+\sum_{i\geq 2}\alpha_it^i$. Using again that $t_1=t+\sum_{i\geq 2}\alpha_it^i$, we have
$$f_{m-1}(y^{(m-1)})=\sum_{i=i_0+1}^{m-2}(h_i(\alpha_1,\dots,\alpha_{m-1-i})+i\alpha_{m-i})f_i(x^{(i)})+f_{m-1}(x^{(m-1)}),$$
where $h_i, i_0+1\leq i\leq m-2,$ are polynomials over $k$.  Inspired by this relation, we define functions $\alpha_{i}(x^{(m-1)})$ as follows.
First, since $m<p$, we can construct by induction the values $\tilde{\alpha}_{1}=1, \tilde{\alpha}_{2}, \dots,\tilde{\alpha}_{m-i_0-2}$ such that $ h_i(\tilde{\alpha}_{1},\dots,\tilde{\alpha}_{m-i-1})+i \tilde{\alpha}_{m-i}=0$ for all $i_0+2\leq i\leq m-2$. Then we set
\begin{align*}
 \alpha_{i}(x^{(m-1)}) &= \tilde{\alpha}_{i} \qquad\text{ for } 1\leq i\leq m-i_0-2 ,  \\
 \alpha_{m-i_0-1}(x^{(m-1)})  &= \frac{1}{i_0+1} \Biggl(\frac{-\frac{b_0}{b_{m-1}}f_0(x^{(0)})-f_{m-1}(x^{(m-1)})}{f_{i_0+1}(x^{(i_0+1)})}   \\
&\qquad
 -h_{i_0+1}(\tilde{\alpha}_{1},\dots,\tilde{\alpha}_{m-i_0-2})\Biggr) , \\
 \alpha_{i}(x^{(m-1)})  &= 0  \qquad\text{ for } i>m-i_0-1.
\end{align*}
The point is that, with these definitions,
$$\bigl(g_{i}(x_{0j},\dots,x_{ij},\alpha_{1}(x^{(m-1)}),\dots,\alpha_{i}(x^{(m-1)}))\bigr)_{0\leq i\leq m-1, 1\leq j\leq n}$$
is mapped to $0$ by $(b_0f_0+b_{m-1}f_{m-1})|_V$.
And then the morphism $\theta:V\to\AA_{k}^1\times V_0$ given by
\begin{align*}
\theta(x^{(m-1)})=&\Bigl[(b_0f_0+b_{m-1}f_{m-1})(x^{(m-1)}),  \\
&\bigl(g_{i}(x_{0j},\dots,x_{ij},\alpha_{1}(x^{(m-1)}),\dots,\alpha_{i}(x^{(m-1)}))\bigr)_{0\leq i\leq m-1, 1\leq j\leq n}\Bigr]
\end{align*}
is indeed well defined, and has an inverse morphism by construction. It is also clear that $\pi\circ\theta=(b_0f_0+b_{m-1}f_{m-1})|_V$.
\end{proof}

\begin{cor}\label{zero1} We use Notation \ref{notation for finite fields}. Suppose that $p>m>i_0+2\geq 2$. Then we have for all $i$ that
$$\H_c^i((Z_{m-1,i_0}\backslash Z_{m-1,i_0+1})\otimes_{\FF_q}\overline{\FF}_p,\cL_{f,b,\Psi})=0 .$$
\end{cor}

\begin{proof}
We remark that $\H_c^i(\AA_{\overline{\FF}_p}^1, \cL_{\Psi})=0$ for all $i$ (see \cite[Lemme 8.5]{DeligneWeil1}). Then the claim follows from Lemma \ref{fibration} and the K\"unneth formula.
\end{proof}

\begin{cor}\label{zero2} We use Notation \ref{notation for finite fields}. Suppose that $p>m\geq2$.

(1) Then we have for all $i$ that
$$\H_c^i((Z^{(m-1)}\backslash Z_{m-1,m-2})\otimes_{\FF_q}\overline{\FF}_p,\cL_{f,b,\Psi})=0$$
and
$$\H_c^i(Z^{(m-1)}\otimes_{\FF_q}\overline{\FF}_p,\cL_{f,b,\Psi})=\H_c^i(Z_{m-1,m-2}\otimes_{\FF_q}\overline{\FF}_p,\cL_{f,b,\Psi}).$$

(2) In particular,  we have that
$$\H_c^i(Z^{(m-1)}\otimes_{\FF_q}\overline{\FF}_p,\cL_{f,b,\Psi})=0$$
if $i>2\dim(Z_{m-1,m-2})$.
\end{cor}

\begin{proof}
(1) The first claim is trivial if $m=2$, and follows for $m>2$ by using Corollary \ref{zero1} and the spectral sequence from \cite[Sommes trig. 2.5*]{Deligneetalecoho}.
The second  claim is implied by the first and by the long exact sequence for the triple $Z^{(m-1)}\backslash Z_{m-1,m-2}\hookrightarrow Z^{(m-1)}\hookleftarrow Z_{m-1,m-2}$ \cite[Sommes trig. 2.5*]{Deligneetalecoho}, namely
\begin{align*}
\dots\rightarrow &\H_c^i((Z^{(m-1)}\backslash Z_{m-1,m-2})\otimes_{\FF_q}\overline{\FF}_p,\cL_{f,b,\Psi})\\
\rightarrow &\H_c^i(Z^{(m-1)}\otimes_{\FF_q}\overline{\FF}_p,\cL_{f,b,\Psi})\\
\rightarrow &\H_c^i(Z_{m-1,m-2}\otimes_{\FF_q}\overline{\FF}_p,\cL_{f,b,\Psi})\rightarrow\dots .
\end{align*}
(2) This is then immediate since
$\H_c^i(Z_{m-1,m-2}\otimes_{\FF_q}\overline{\FF}_p,\cL_{f,b,\Psi})=0$
if $i>2\dim(Z_{m-1,m-2})$.
\end{proof}

For later use, we mention a  special fibration property when $f$ is smooth.

\begin{lem}\label{smooth} We use Notation \ref{notation for finite fields}.
Suppose that $f$ is smooth at $Z$. Then there exists a stratification of $Z_{k,m-1,m-2}$ into $n$ locally closed subsets,  such that the restriction of $b_0 f_{0}+b_{m-1}f_{m-1}$ to each subset is a trivial fibration.
\end{lem}

\begin{proof}
Let $x^{(m-1)}\in Z_{k,m-1,m-2}$ and $x\in (k[[t]]/(t^m))^n$ the associated jet. In particular, $y^{(m-1)}\in  Z_{k,m-1,m-2}$ if $y\in x+ t^{m-1}k[[t]]^n$. By a simple calculation we have
\begin{align*}
f(x+t^{m-1}z)&=f_0(x^{(0)}) + f_{m-1}(x^{(m-1)})t^{m-1} \\
&+\sum_{j=1}^{n}\frac{\partial f}{\partial x_j}(x_{01},x_{02},\dots,x_{0n})t^{m-1}z_{0j} \mod t^m
\end{align*}
for $z\in k[[t]]^n$. We take
\begin{align*}
A_j:=&\{x^{(m-1)}\in Z_{k,m-1,m-2}\mid \frac{\partial f}{\partial x_i}(x_{01},x_{02},\dots,x_{0n})=0 \quad\text{for all } i<j, \\
&\text{and } \frac{\partial f}{\partial x_j}(x_{01},x_{02},\dots,x_{0n})\neq 0\}.
\end{align*}
Since $f$ is smooth at $Z$, it is clear that the restriction of $b_0 f_{0}+ b_{m-1}f_{m-1}$ to $A_j$ is a trivial fibration for all $j$.
\end{proof}

\section{Link with characteristic zero}\label{section5}

Here we link the exponential sums in positive characteristic of the previous section to our original exponential sums through a transfer principle.
In particular, we then show that a conjectural relation between the highest weight and the motivic oscillation index has important corollaries.
We first generalize some notation and terminology of Section \ref{section2} to non-archimedean fields of arbitrary characteristic.

\smallskip
Recall that, for a positive integer $M$, we denote by $\cL_M$ the set of all non-archimedean local fields of characteristic $0$ with residue field characteristic at least $M$, and by $\cL_{K,M}$ the set of fields in $\cL_M$ endowed with an embedding of $K$.
We denote similarly by $\cL'_M$ the set of all  non-archimedean local fields of characteristic at least $M$,
and by  $\cL'_{K,M}$ the set of fields in  $\cL'_M$ which are endowed with a structure of  $\cO_K$-algebra.  (Note that this is consistent; in characteristic zero a structure of $\cO_K$-algebra implies an embedding of $K$.)

For example, in the basic case $K=\QQ$, we have that $\cL'_{M}=\cL'_{K,M}$ consists of the fields $\FF_q((t))$, with characteristic of $\FF_q$ at least $M$.

\begin{notation}\label{notation2}
Let $f\in\overline{\QQ}[x_1,\dots,x_n]$ be a non-constant polynomial and $Z$ a reduced subscheme of $\AA_{\overline{\QQ}}^n$.
Let $\cO$ be a finitely generated $\ZZ$-subalgebra of $\overline{\QQ}$ over which $f$ and $Z$ are defined.
With the latter we mean that $Z$ is induced by a subscheme of $\AA_{\cO}^n$, that we denote also by $Z$.

(1) We say that a non-archimedean field $L$ {\em is over $\cO$} if $\cO_L$ is endowed with an $\cO$-algebra structure $\varphi:\cO\to \cO_L$. We denote the induced morphism between polynomial rings over these rings also by $\varphi$.
If there is no confusion about $\varphi$, we still denote $f$ for $\varphi(f)\in \cO_L[x_1,\dots,x_n]$ and $f_{k_L}$ or $f\mod \cM_L$ for the image of $\varphi(f)$ in $k_L[x_1,\dots,x_n]$.
For such $L$ we define the Schwartz-Bruhat function $\Phi_{L,Z}:=\11_{\{x\in\cO_L^n| \overline{x}\in Z(k_L)\}}$.

(2) Let $c_1,\dots,c_r$ be the critical values of $f$, these are elements of $\overline{\QQ}$. Put $\cO_i=\cO[c_i]$ and let $Z_i$ be the intersection of $Z\otimes_\cO\cO_i$ with $\{f=c_i\}$ in $\AA^n_{\cO_i}$.
\end{notation}

\subsection{A useful property of weights}

The following facts are easily verified (for example, by
model theoretic compactness).

\begin{lem}\label{critical values}
Fix a number field $K \supset \cO$.
For each $i=1,\dots,r$ we denote by $g_i\in K[y]$ the minimal polynomial of $c_i$ over $K$. There exists a constant $M$, depending only on $f$, such that for all $L\in \cL_{K,M}\cup\cL'_{K,M}$ the following facts hold:
\begin{itemize}
\item[(i)]$L$ is over $\cO$ with structure map $\varphi:\cO\to\cO_L$, that extends to a map $\varphi: \cO[c_1,\dots,c_r]\to \overline{\cO}_L$, where $\overline{\cO}_L$ is the integral closure of $\cO_L$ in $\overline{L}$,
\item[(ii)]$\varphi(c_i)$ is integral over $\cO_L$ for all $i=1,\dots,r$,  hence
$\varphi(g_i)\in \cO_L[y]$ and $\varphi(g_i) \mod \cM_L$ has only simple roots over $\overline{k}_L$,
\item[(iii)]if $i\neq j$ and $L$ is over $\cO_i$ and $\cO_j$ with the structure maps induced by $\varphi$,  then $\Supp(\Phi_{L,Z_i})\cap\Supp(\Phi_{L,Z_j})=\emptyset$,
\item[(iv)]the critical values of  $f \mod \cM_L$ over $\overline{k}_L$ are $\overline{\varphi (c_i)}, 1\leq i\leq r$ (the notation $\overline{\varphi(c_i)}$ is well-defined by (i) and (ii)).
\end{itemize}
\end{lem}

\begin{prop}\label{reduction} Take $f$ and $Z$  as in Notation \ref{notation2}.
For all $m\geq 2$, there exists a number $M_m$, depending on $m$,  such that for all local fields $L=k_L((t))\in\cL'_{K,M_m}$ and for all $(b_0,b_{m-1})\in k_L\times k_L^{*}$,  we have
\begin{align*}
w(b_0f_{k_L,0}+b_{m-1}f_{k_L,m-1},Z_{k_L}^{(m-1)})&=w(b_{m-1}f_{k_L,m-1},Z_{k_L}^{m-1})\\
&=\max_{i\in I_L}w(b_{m-1}f_{k_L,m-1},Z_{i,k_L}^{(m-1)})\\
&=\max_{i\in I_L}w(b_0f_{k_L,0}+b_{m-1}f_{k_L,m-1},Z_{i,k_L}^{(m-1)})
\end{align*}
where $I_L=\{1\leq i\leq r|\varphi(\cO_i)\subset L\}$.
\end{prop}

\begin{proof}
By Lemma \ref{critical values}, we have that $f_{k_L}$ has no critical point in the complement $U$ of the closed subscheme  $\sqcup_{i\in I_L}Z_{i,k_L}$ in $Z_{k_L}$. By Corollary \ref{zero2} and Lemma \ref{smooth}, we have for all $m\geq 2$, all $L\in \cL'_{K,M_m}$,  all $j$, and all $b=(b_0,b_{m-1})\in k_L\times k_L^*$ that
$$\H_c^j(U^{(m-1)}\otimes_{k_L}\overline{k}_L, \cL_{f_{k_L},b,\Psi})=0,$$
 hence
$$\H_c^j(Z_{k_L}^{(m-1)}\otimes_{k_L}\overline{k}_L, \cL_{f_{k_L},b,\Psi})=\oplus_{i\in I_L}\H_c^j(Z_{i,k_L}^{(m-1)}\otimes_{k_L}\overline{k}_L, \cL_{f_{k_L},b,\Psi}),$$
and consequently
\begin{equation}\label{b0}
w(b_0f_{k_L,0}+b_{m-1}f_{k_L,m-1},Z_{k_L}^{(m-1)})=\max_{i\in I_L}w(b_0f_{k_L,0}+b_{m-1}f_{k_L,m-1},Z_{i,k_L}^{(m-1)})
\end{equation}
and
\begin{equation}\label{0b0}
w(b_{m-1}f_{k_L,m-1},Z_{k_L}^{(m-1)})=\max_{i\in I_L}w(b_{m-1}f_{k_L,m-1},Z_{i,k_L}^{(m-1)}).
\end{equation}
By the definition of $Z_i$ we have that $f_{k_L,0}$ is constant on $Z_{i,k_L}^{(m-1)}$ for all $i\in I_L$.  Using the K\"unneth formula for $Z_{i,k_L}^{(m-1)}\times \Spec(k_L)$, the function $b_{m-1} f_{k_L, m-1}$ on $Z_{k_L}$ and a constant function on $\Spec(k_L)$, we then obtain
\begin{equation}\label{cons}
w(b_0f_{k_L,0}+b_{m-1}f_{k_L,m-1},Z_{i,k_L}^{(m-1)})=w(b_{m-1}f_{k_L,m-1},Z_{i,k_L}^{(m-1)})
\end{equation}
for all $m\geq 2$,  all $L\in \cL'_{K,M_m}$, all $b=(b_0,b_{m-1})\in k_L\times k_L^*$, and all $i\in I_L$. Our claim follows by combining (\ref{b0}), (\ref{0b0}) and (\ref{cons}).
\end{proof}

\subsection{Transfer principle}

We show the following identification of exponential sums over local fields in $\cL_{K,M}\cup\cL'_{K,M}$ with the same residue field.

\begin{prop}[Transfer principle]\label{transfer} Let $f\in \overline{\QQ}[x_1,\dots,x_n]$ be a non-constant polynomial and $Z$ a subscheme of $\AA_{\overline{\QQ}}^n$ such that $f$ and $Z$ are defined over a number field $K$. We use Notation \ref{notation2}. There exists a number $M$, depending only on $f$, such that for all pairs of local fields $(L,\pi,\psi)$ and $(L',\pi',\psi')$ in  $\cL_{K,M}\cup\cL'_{K,M}$ over $\cO$ with $k_L\simeq k_{L'}$, and for all $z\in L$ and $z'\in L'$ with $\ord_{L}(z)=\ord_{L'}(z')\leq -1$ and $\ac(z)=\ac(z')$, we have
\begin{equation}\label{residuefield}
E_{f,L,Z,\psi_z}=E_{f,L',Z,\psi'_{z'}} .
\end{equation}
\end{prop}
\begin{proof} The claim  is trivial if $\ord_L(z)=\ord_{L'}(z')=-1$, so we can suppose that $\ord_L(z)=\ord_{L'}(z')\leq -2$.
Firstly, we suppose that $f(Z(\CC))=0$. We remark that there exists a number $M$, depending only on $f$, such that for all $L$ in $\cL_{K,M}\cup\cL'_{K,M}$ over $\cO$, for all $z\in L$ with $\ord_{L}(z)\leq -2$ and all additive characters $\psi$ of conductor $0$, we have
\begin{equation}\label{m-2}
\int_{\{x\in\cO_L^n \mid \overline{x}\in Z(k_L), \ord_L(f(x))\leq -\ord_L(z)-2\}}\psi(zf(x))|dx|=0.
\end{equation}
We can prove this by adapting the argument from \cite[Lemma 3.2.1]{Saskia-Kien}. (A weaker version of (\ref{m-2}) with $M$ depending on $f,       \ord_L(z)$ and $L\in \cL'_{K,M}$  follows from Corollary \ref{zero2}.) Hence, for all $(L,\pi,\psi)\in\cL_{K,M}\cup\cL'_{K,M}$  over $\cO$ and for all $z\in L\setminus\cO_L$, we have
\begin{align*}
&\int_{\{x\in\cO_L^n|\overline{x}\in Z(k_L)\}}\psi(zf(x))|dx|\\
=&\int_{\{x\in\cO_L^n|\overline{x}\in Z(k_L), \ord_L(f(x))\geq-\ord_L(z)-1\}}   \psi(zf(x))|dx|\\
=&\sum_{\alpha\in k_L}F_L(\alpha,-\ord_L(z))\Psi_{p_L}\bigl(\Tr_{k_L/\FF_{p_L}}(\alpha\ac(z))\bigr) ,
\end{align*}
where the last line has the following meaning.
In the terminology of \cite{Cluckers-Loeser-Inv}, $F$ is a constructible motivic function on $\RF\times \VG$, where $\RF$ and $\VG$ stand for residue field variables and value group variables, respectively. More precisely,
$$F_L(\alpha,m)=\operatorname{Vol}\{x\in\cO_L^n \mid \overline{x}\in Z(k_L), \ord_L(f(x))=m-1, \ac(f(x))=\alpha\}$$
 if $\alpha\neq 0$ and
 $$F_L(0,m)=\operatorname{Vol}\{x\in\cO_L^n \mid \overline{x}\in Z(k_L), \ord_L(f(x))>m-1\}$$
for all $L\in \cL_{K,M}\cup\cL'_{K,M}$.
We can write  $F$ as a linear combination of basic constructible motivic functions on $\RF\times\VG$,  so the interpretation of constructible motivic functions on local fields $L\in\cL_{K,M}\cup\cL'_{K,M}$ for large enough $M$ (see \cite{Cluckers-Loeser-ann}) implies (\ref{residuefield}).

\medskip
For general $Z$, we can conclude from Lemma \ref{critical values} that there is a partition
\begin{equation}\label{decomposition}
\Supp(\Phi_{L,Z})=\left(\bigsqcup_{1\leq i\leq r, \varphi(\cO_i)\subset L}\Supp(\Phi_{L,Z_i})\right)\bigsqcup\Supp(\Phi_{L,Z_0})
\end{equation}
of $\Supp(\Phi_{L,Z})$, where $Z_0$ is a subscheme  of $Z$
such that $f \mod \cM_L$ has no critical point in $\overline{\Supp(\Phi_{L,Z_0})}$. Since moreover $\Phi_{L,Z_0}$ is residual, we derive that $E_{f,L,Z_0,\psi_z}=0$ for all $L\in\cL_{K,M}\cup\cL'_{K,M}$ and all $z\in L$ with $\ord_L(z)\leq -2$. So our claim follows by applying the above discussion for $(f-c_i, Z_i), 1\leq i\leq r,$ and using the fact that
$$E_{f,L,Z,\psi_z}=\sum_{1\leq i\leq r, \varphi(\cO_i)\subset L} \psi_z(\varphi(c_i)) \cdot  E_{f-c_i,L,Z_i,\psi_z}$$
for all  $L\in\cL_{K,M}\cup\cL'_{K,M}$ and all $z\in L$ with $\ord_L(z)\leq -2$.
\end{proof}

\begin{cor}\label{transfer1}  Let $f$ and $\cO$ be as in Notation \ref{notation2} and $K\supset \cO$ a number field.   Let $(W_i)_{i\in I}$ be a family of subschemes of $\AA_{\cO}^n$ and $(\sigma_i)_{i\in I}$  positive real numbers. Assume that $f(W_i(\CC))$ contains at most one critical value of $f$. Suppose that for each $m\geq 2$ there exist constants $M_m$ and $C_m$ such that for all local fields $L\in \cL'_{K,M_m}$,  for all additive characters $\psi$ of $L$ with $m_\psi=m$, and for all $i\in I$, we have
$$|E_{f,L,W_i,\psi}|\leq C_mq_L^{-m\sigma_i}.$$
Then there exist constants $M$ and $C$ such that for all $m\geq 2$, for all local fields $L\in \cL_{K,M}\cup\cL'_{K,M}$, for all additive characters $\psi$ of $L$ with $m_\psi=m$, and for all $i\in I$, we have
$$|E_{f,L,W_i,\psi}|\leq Cm^{n-1}q_L^{-m\sigma_i}.$$
\end{cor}

\begin{proof}
The proof follows by adapting \cite[Claim 3.2.7]{Saskia-Kien}, in combination with Proposition \ref{transfer}.
\end{proof}

\begin{remark}
In Corollary \ref{transfer1}, the condition that $f(W_i(\CC))$ contains at most one critical value of $f$ may not be necessary, but to prove it more work is needed, combining \cite[Claim 3.2.7]{Saskia-Kien} and \cite[Proposition 2.7]{Denef-Veys}. However, almost every  problem can be reduced to Corollary \ref{transfer1} by using Proposition \ref{transfer}, therefore we do not explore this here.
\end{remark}

\subsection{A weight conjecture and applications}

Motivated by Conjecture \ref{Conjexp} and Proposition \ref{transfer}, we propose the following conjecture.  In fact, we will prove below that it implies
 Conjecture \ref{Conjexp} and Conjecture \ref{Conj1}.

\begin{conj}\label{highest}  Let $f$, $Z$ and $\cO$ be as in Notation \ref{notation2} and $K\supset \cO$ a number field.  Then for each $m\geq 2$ there exists  a number $M$,  depending on $m$, such that for all local fields $L\in\cL'_{K,M}$, we have that $w(f_{k_L},b,Z_{k_L}^{(m-1)}) = w(b_0f_{k_L,0}+b_{m-1}f_{k_L,m-1},Z_{k_L}^{(m-1)})\leq 2(mn-m \moi_K(f,Z))$ for all $(b_0, b_{m-1})\in k_L \times k_L^*$, where $Z_{k_L}=Z\otimes_{\cO} k_L$.
\end{conj}

\begin{remark}
We may think about a number $M$ that does not depend  on $m$. But this is not necessary for our purposes, as we will see below, and it could be very hard if $M<m$.
\end{remark}

\begin{remark}
The motivic oscillation index is a quite abstract notion, and for the moment we do not have a useful tool to compute it in general.
For arithmetic applications, it is sometimes good enough to prove an analogue of Conjecture \ref{highest}, with $\moi_K(f,Z)$ replaced by some \lq more accessible\rq\ other invariant that is a lower bound for $\moi_K(f,Z)$. Hence, such analogues are also interesting to study. In fact, we will show the variant with the log canonical threshold.

Recall that, when $f(Z(\CC))=0$, the {\em minimal exponent $\tilde{\alpha}_{Z}(f)$ of $f$ at $Z$} is the smallest number among all values $\tilde{\alpha}_{P}(f)$ with $P\in Z(\CC)$, where $\tilde{\alpha}_{P}(f)$ is the minimal exponent of $f$ at $P$, i.e., $-\tilde{\alpha}_{P}(f)$ is the largest non-trivial root of the Bernstein-Sato polynomial of $f$ at $P$.
A famous conjecture (sometimes called the strong monodromy conjecture) predicts essentially that poles of Igusa zeta functions of $f$ are always roots of the Bernstein-Sato polynomial of $f$.
Therefore, $\tilde{\alpha}_{Z}(f)$ is a \lq predicted\rq\ lower bound for $\moi_K(f,Z)$.

A conceptually appealing hope, when $f(Z(\CC))=0$, is that moreover $\tilde{\alpha}_{Z}(f)= \moi_K(f,Z)$, which then would automatically yield Conjecture \ref{Conj1} (since the Bernstein-Sato polynomial is invariant under base field extension).
\end{remark}

Firstly, we check Conjecture \ref{highest} when $f$ is smooth at $Z_\CC$. In this case $\moi_K(f,Z)=+\infty$, so we need to show that there exists a value $M$ such that $w(f_{k_L},b,Z_{k_L}^{(m-1)}) = w( b_0 f_{k_L,0} + b_{m-1}f_{k_L,m-1},Z_{k_L}^{(m-1)})=-\infty$ for all  $L\in\cL'_{K,M}$, and for all $(b_0, b_{m-1})\in k_L \times k_L^*$. If $f$ is smooth at $Z_\CC$, then by Robinson's principle or by \cite[Corollary 2.2.10]{Marker}, there exists $M$ such that $f_{k_L}$ is smooth at $Z_{k_L}$ for all $L\in \cL'_{K,M}$. We conclude by using the same idea as in the proof of Corollary \ref{zero2} and Lemma \ref{smooth}.

\begin{prop}\label{Implication} If Conjecture \ref{highest} holds for $(f, Z)$, then

(1)  Conjecture \ref{Conjexp} holds for $(f, Z)$,

(2) Conjecture \ref{Conj1} holds for $(f, Z)$.
\end{prop}

\begin{proof}    By Proposition \ref{reduction}, the proof of Proposition \ref{transfer} and the definition of $\moi_K(f,Z)$, it suffices to prove our proposition in case $f(Z(\CC))=0$.

(1)  Using \cite{Katz}, for each $m\geq 2$, there exist a number $M_m$ and a constant $B_m$ (the Bombieri constant), both depending on $m$, $\deg(f)$ and the complexity of $Z$, such that for all local fields $L\in\cL'_{K,M_m}$ we have
$$\sum_{i}\dim_{\overline{\QQ}_\ell}\H^i_c(Z_{k_L}^{(m-1)}\otimes_{\FF_q}\overline{\FF}_p,\cL_{f_{k_L}, b, \Psi})\leq B_m.$$
So for each $m\geq 2$, there exists a number $M_m$ such that, for all local fields $(L,\pi,\psi)\in\cL'_{K,M_m}$ and for all $z\in L$ with $\ord_z=-m$, we have
$$|E_{f,L,Z,\psi_z}|\leq B_{m}q_L^{-m\moi_K(f,Z)}.$$
By Corollary \ref{transfer1}, there exist  constants $C,M$ such that, for all local fields $(L,\pi,\psi)\in\cL_{K,M}\cup\cL'_{K,M}$ and  for all $z\in L$ with $\ord_L(z)\leq -2$, we have
$$|E_{f,L,Z,\psi_z}|\leq C|\ord_L(z)|^{n-1}q_L^{\ord_L(z)\moi_K(f,Z)}.$$

(2) If $w(b_0f_{k_L,0}+b_{m-1}f_{k_L,m-1},Z_{k_L}^{(m-1)})\leq 2(mn-m\sigma)$ for all $m\geq 2$,  all $L\in\cL'_{K,M_m}$, and all $(b_0,b_{m-1})\in k_L\times k_L^*$, then by the above argument and the definition of $\moi_K(f,Z)$ we have that $\moi_K(f,Z)\geq \sigma$. For a finite extension  $K'$ of $K$, since the weights are stable by  field extension, Conjecture \ref{highest} implies that
$$w(b_0 f_{k_L,0}+b_{m-1}f_{k_L,m-1},Z_{k_L}^{(m-1)})\leq 2(mn-m\moi_{K'}(f,Z))$$
for all $m\geq 2$, all $L\in\cL'_{K,M}$, and all $(b_0,b_{m-1})\in k_L\times k_L^*$ . So we have $\moi_K(f,Z)\geq\moi_{K'}(f,Z)$. However, $\moi_K(f,Z)\leq\moi_{K'}(f,Z)$ by definition, and thus $\moi_K(f,Z)=\moi_{K'}(f,Z)$.
\end{proof}

\begin{cor}\label{uniformZ}  Let $f$, and $\cO$ be as in Notation \ref{notation2} and $K\supset \cO$ a number field. Let $(W_i)_{i\in I}$ be a family of subschemes of $\AA_{\cO}^n$ with bounded complexity. Suppose that for each $m\geq 2$ there exists a constant $M_m$ such that Conjecture \ref{highest} holds  for all $L\in\cL'_{K,M_m}$, all $b\in k_L\times k_L^*$ and all $W_i, i\in I$. Then there exist  constants $M, C$ such that for all local fields $L\in \cL'_{K,M}\cup \cL_{K,M}$,  all additive characters $\psi$ of $L$ with $m_\psi\geq 2$, and all $i\in I$, we have
$$|E_{f,L,W_i,\psi}|\leq Cm_\psi^{n-1}q_L^{-m_\psi\moi_K(f,W_i)}.$$
\end{cor}

\begin{proof} This follows by Corollary \ref{transfer1} and the proof of Proposition \ref{Implication}.
\end{proof}

\section{Unconditional results}\label{section6}

\subsection{Non-rational singularities}

Applying the main result of \cite{CMN} in the presence of non-rational singularities, we prove Conjecture \ref{highest} in that setting.

\begin{thm}\label{nonrational} Let $f$, $Z$ and $\cO$ be as in Notation \ref{notation2} and $K\supset \cO$ a number field.  If, for some critical value $c_i$ of $f$ contained in $f(Z(\CC))$, we have that $f-c_i$ has non-rational singularities in every neighbourhood of $Z(\CC)$ in $\CC^n$, then Conjecture \ref{highest}  holds for $(f,Z)$.
\end{thm}

\begin{proof}
Let us fix $m\geq 2$. We use the notion  log canonical threshold $\lct_Z(f)$ of $f$ around $Z$ as in \cite{MustataJAMS}. Let $c_i$ be a critical value of $f$ contained in $f(Z(\CC))$.  By \cite[Proposition 3.10]{CMN},  $\moi_{K(c_i)}(f-c_i,Z_i)>1=\lct_{Z_i}(f-c_i)$ if $f-c_i$ has rational singularities in some neighborhood of $Z_i(\CC)$ in $\CC^n$, and $\moi_{K(c_i)}(f-c_i,Z_i)=\lct_{Z_i}(f-c_i)\leq 1$ if $f-c_i$ has non-rational singularities in every neighborhood of $Z_i(\CC)$ in $\CC^n$. Hence, by Proposition \ref{reduction} and the definition of $\moi_K(f,Z)$, it suffices to prove the weaker version of Conjecture \ref{highest} with $\lct_Z(f)$ instead of $\moi_K(f,Z)$,  moreover assuming that $f(Z(\CC))=0$.

\smallskip
Firstly, we show the following claim.

\noindent
 {\em There exists $M_m$ such that, for all local fields $L\in\cL'_{K,M_m}$, we have
$$\H_c^i(Z_{k_L}^{(m-1)}\otimes_{\FF_q}\overline{\FF}_p,\cL_{f_{k_L},b,\Psi})=0$$
 for all $i\geq 2(mn-(m-1)\lct_Z(f))$.}

\smallskip
\noindent
We combine several facts.

(i) By Robinson's principle, there exists $M_m>m$ such that the dimension of $Z_{k_L,m-1,m-2}$ does not depend on $L\in \cL'_{K,M_m}$ and is equal to the dimension of $Z_{\CC,m-1,m-2}$.

(ii) The main result from \cite{CMN} says that a weaker version of Conjecture \ref{Conjexp} with $\lct_Z(f)$ instead of $\moi_K(f,Z)$ holds for all $L\in\cL_{K,M_m}$.

(iii) Using the description of the log canonical threshold in terms of jet schemes in \cite{MustataJAMS}, we deduce that $\dim(Z_{k_L,m-1,m-2})\leq mn-(m-1)\lct_Z(f)$.

\smallskip
Then Proposition \ref{transfer} in $\cL_{K,M_m}\cup\cL'_{K,M_m}$ and  Corollary \ref{zero2}  imply that
$$\H_c^i(Z_{k_L}^{(m-1)}\otimes_{\FF_q}\overline{\FF}_p,\cL_{f_{k_L},b,\Psi})=0$$
for $L\in\cL'_{K,M_m}$ if $i>2(mn-(m-1)\lct_Z(f))$.
The only case left to prove is
\begin{equation}\label{true}
\H_c^{2(mn-(m-1)\lct_Z(f) )} (Z_{k_L,m-1,m-2}\otimes_{\FF_q}\overline{\FF}_p,\cL_{f_{k_L},b,\Psi})=0.
\end{equation}
If $\dim(Z_{k_L,m-1,m-2})<mn-(m-1)\lct_Z(f)$, then the equality (\ref{true}) is true. Suppose that $\dim(Z_{k_L,m-1,m-2})=mn-(m-1)\lct_Z(f)$ and (\ref{true}) does {\em not} hold for $L\in\cL'_{K,M_m}$. By the argument in \cite[Remark. 1.18]{Deligneetalecoho} for Poincar\'e duality or the more detailed argument in \cite[Example 20]{Kowalski}, there exists a non-empty open subset $U$ of $Z_{k_L,m-1,m-2}$ such that the trace function $\Tr_{k/\FF_{p_L}}((b_0f_{k_L,0}+b_{m-1}f_{k_L,m-1})(x^{(m-1)}))$ is constant on $U(k)$ for all finite extensions $k$ of $k_L$. So it suffices to look at the Grothendieck trace formula for the constant $\QQ_\ell$-sheaf of rank $1$ on $U$ when we want to find the size of the  exponential sum related to $b_0f_{k_L,0}+b_{m-1}f_{k_L,m-1}$ on $U$. We then use \cite[Theorem. 3.1]{Bombieri-Katz}, the majoration for the sum of Betti numbers in \cite{Katz}, the Grothendieck trace formula and Corollary \ref{zero2} to conclude that, for each $\epsilon>0$, there exists a constant $C_{\epsilon,m}$ such that
\begin{equation}\label{ineq1}
|E_{f,L',Z,\psi_{a\pi'^{-m}}}|\geq q_{L'}^{-mn}C_{\epsilon,m}q_{L'}^{mn-(m-1)\lct_Z(f)-\epsilon}
\end{equation}
for infinitely many finite extensions $(L',\pi',\psi')$ of $(L,\pi,\psi)$, and for some $a\in\cO_{L'}^*$ (depending on $b_{m-1}$).

On the other hand, from \cite[Theorem 1.5]{CMN}, Proposition \ref{transfer} and the fact that $f$ has non-rational singularities in every neighbourhood of $Z(\CC)$, we deduce that Conjecture \ref{Conjexp} holds for all such local fields $L'$, and  therefore
\begin{equation}\label{ineq2}
|E_{f,L',Z,\psi_{a\pi'^{-m}}}|<C_mq_{L'}^{-m\lct_Z(f)}
\end{equation}
for some constant $C_m>0$ for all $a\in\cO_{L'}^*$. We obtain a contradiction if we take $0<\epsilon\ll\lct_Z(f)$ and $q_{L'}\to +\infty$. Hence, we conclude that  equality (\ref{true}) holds for all $L\in\cL'_{K,M_m}$, thereby finishing the proof of the claim.

\smallskip
Using Corollary \ref{zero2} again, we deduce that
$$w(f_{k_L},b,Z_{k_L}^{(m-1)})  \leq 2(mn-(m-1)\lct_Z(f))-1$$
for all $L\in\cL'_{K,M_m}$ and  all $(b_0, b_{m-1})\in k_L \times k_L^*$. If $\moi_K(f,Z)\leq 1/2$, then we have $w(f_{k_L},b,Z_{k_L}^{(m-1)})  \leq 2(mn-m\lct_Z(f))$. If $\lct_Z(f)> 1/2$ and $w(f_{k_L},b,Z_{k_L}^{(m-1)})  = 2(mn-(m-1)\lct_Z(f) )-1$, by using the same argument as above, we see that there exists a constant $C_{\epsilon,m}$  such that
$$|E_{f,L',Z,\psi_{a\pi'^{-m}}}|\geq q_{L'}^{-mn}C_{\epsilon,m}q_{L'}^{mn-(m-1)\lct_Z(f)-1/2-\epsilon}$$
for infinitely many finite extensions $(L',\pi',\psi')$ of $(L,\pi,\psi)$, for some $a\in\cO_{L'}^*$.
We have a contradiction with inequality (\ref{ineq2}) if we take $0<\epsilon\ll \lct_Z(f)-1/2$ and $q_{L'}\to+\infty$. So, if $\lct_Z(f)> 1/2$, then
$$ w(f_{k_L},b,Z_{k_L}^{(m-1)})  \leq 2(mn-(m-1) \lct_Z(f))-2 .$$
We finish the proof by observing that $\lct_Z(f)\leq 1$.
\end{proof}

Note that the proof above of Theorem \ref{nonrational} implies the variant of Conjecture \ref{highest} with the log canonical threshold.

\begin{prop} Let $f$, $Z$ and $\cO$ be as in Notation \ref{notation2} and $K\supset \cO$ a number field.  Then for each $m\geq 2$ there exists  a number $M$,  depending on $m$, such that for all local fields $L\in\cL'_{K,M}$, we have that $w(b_0f_{k_L,0}+b_{m-1}f_{k_L,m-1},Z_{k_L}^{(m-1)})\leq 2(mn-m \sigma_{Z}(f))$ for all $(b_0, b_{m-1})\in k_L \times k_L^*$, where $\sigma_{Z}(f)=\min_{P\in Z(\CC)}\lct_P(f-f(P))$.
\end{prop}

\subsection{Weight conjecture and similarity}

\begin{thm}\label{transferanalytic} Let $f$, $Z$ and $\cO$ be as in Notation \ref{notation2} and $K\supset \cO$ a number field. Let $P,Q\in \cO^n$ satisfying $f(P)=g(Q)=0$. If $(f,P)\sim_K (g,Q)$, then Conjecture \ref{highest} holds for $(f,P)$ if and only if it holds for $(g,Q)$.
\end{thm}

\begin{proof}
As in the proof of Main Theorem \ref{mainthm1}, we can assume that $P=Q=0$. Denote $Z=\{0\}\subset\AA^n$. Since $(f,0)\sim_K (g,0)$, there exists neighbourhoods $U,V$ of $0$ in $\CC^n$, a biholomorphic map $\theta:U\to V$ and a holomorphic function $v: V\to \CC^*$ such that $f=(g.v)\circ\theta$.

We take local coordinates $(x_1,\dots,x_n)$ on $U$ and $(y_1,\dots,y_n)$ on $V$. 
Let $\theta$ and $v\circ \theta$ be given in coordinates as
$$  \theta : U\to V: (x_i)_{1\leq i \leq n}\mapsto \Bigl(y_j = \sum_{I\in\NN^n,|I|\geq 1}a_{j,I}x^I\Bigr) _{1\leq j \leq n},$$
where $\det(a_{j,I})_{1\leq j\leq n, |I|=1}\neq 0$, and
$$  v\circ \theta : U \to  \CC^* : (x_i)_{1\leq i \leq n}\mapsto \sum_{I\in\NN^n}b_{I}x^I ,$$
where $b_{0}\neq 0$, respectively.

For $x=(x_1,\dots,x_n)\in t\CC[[t]]^n$, with $x_j=\sum_{i\geq 1}x_{ij}t^i$, we denote as usual $x^{(i)}=(x_{\ell j})_{1\leq \ell\leq i, 1\leq j\leq n}$.
Then
$$y_j(x)=\sum_{i\geq 1}y_{ij}(x^{(i)})t^i \quad\text{ and } \quad  (v\circ\theta)(x)=b_{0}+\sum_{i\geq 1}v_i(x^{(i)})t^i ,$$
where $y_{ij}$ and $v_i$ are  polynomials in $x^{(i)}$ with coefficients in $\CC$.

Because $\theta$ is biholomorphic, the map $\theta_{m-1}$ given by $x^{(m-1)}\mapsto y^{(m-1)}=(y_{\ell j}(x^{(\ell)}))_{1\leq \ell\leq m-1, 1\leq j\leq n}$ is an automorphism of $Z_\CC^{(m-1)}$ for all $m\geq 2$. Since $f=(g.v)\circ\theta$, we have
$$f_{m-1}(x^{(m-1)})=\sum_{i=1}^{m-2}g_i(y^{(i)})v_{m-1-i}(x^{(m-1-i)})+b_0g_{m-1}(y^{(m-1)})$$
for all $x\in t\CC[[t]]^n$ and $m\geq 2$.  Hence $$\Theta_{m-1}:=\theta_{m-1}|_{Z_{\CC,f,m-1,m-2}}:Z_{\CC,f,m-1,m-2}\to Z_{\CC,g,m-1,m-2}$$ is an isomorphism  for all $m\geq 2$, and moreover  $$f_{m-1}|_{Z_{\CC,f,m-1,m-2}}=b_0 g_{m-1}|_{Z_{\CC,g,m-1,m-2}}\circ\Theta_{m-1}.$$

Remember that $f_{m-1}$ and $g_{m-1}$ have coefficients in $K$. Then, by quantifier elimination for the theory of algebraically closed fields of characteristic $0$ (see \cite[Theorem 3.2.2]{Marker}) or by using \cite[Corollary 2.2.10]{Marker}, there exists $a\in \overline{K}$ and a $\overline{K}$-isomorphism $\tilde{\Theta}_{m-1}:Z_{\overline{K},f,m-1,m-2}\to Z_{\overline{K},g,m-1,m-2}$ such that $f_{m-1}|_{Z_{\overline{K},f,m-1,m-2}}=a g_{m-1}|_{Z_{\overline{K},g,m-1,m-2}}\circ\tilde{\Theta}_{m-1}$.

Let $K'$ be a finite extension of $K$ such that $a\in K'$ and $\tilde{\Theta}_{m-1}$ and its inverse can be defined over $K'$. By Robinson's principle  or by using \cite[Corollary 2.2.10]{Marker} again, there exists $M$ such that $a\in \cO_L^{*}$ and $\tilde{\Theta}_{m-1}$ and its inverse can be defined over $\cO_L$ for all $L\in\cL_{K',M}$.
Moreover,  we then have for all $L\in\cL'_{K',M}$ that $\tilde{\Theta}_{m-1,k_L}$ is an isomorphism between $Z_{k_L,f_{k_L},m-1,m-2}$ and $Z_{k_L,g_{k_L},m-1,m-2}$ and that $$f_{k_L,m-1}|_{Z_{k_L,f_{k_L},m-1,m-2}}=\overline{a} g_{k_L,m-1}|_{Z_{k_L,g_{k_L},m-1,m-2}}\circ\tilde{\Theta}_{m-1,k_L}.$$

\smallskip
After these preparations, we now suppose that Conjecture \ref{highest} holds for $(g,Z)$.
Because the weights do not change by base field extension, we have by definition of $\cL_{f_{k_L},b,\Psi}$ and $\cL_{g_{k_L},b,\Psi}$ that $w(f_L,b, Z_{k_L,f_{k_L},m-1,m-2})=w(\overline{a} g_{k_L},b,Z_{k_{L'},g_{k_{L}},m-1,m-2})$ for all $L\in\cL'_{K,M}$ and  $L'\in\cL'_{K',M}$ such that $L'$ is a finite extension of $L$, and  all $b=(b_0,b_{m-1})\in k_L\times k_L^{*}$. By Corollary \ref{zero2} and the fact that Conjecture \ref{highest} holds for $(g,Z)$, we must have $$w(f_{k_L},b,Z_{k_L,f_{k_L},m-1,m-2})\leq 2(mn-m\moi_{K'}(g,Z)).$$

By the proof of Proposition \ref{Implication}, we have $\moi_K(g,Z)=\moi_{K'}(g,Z)\leq\moi_K(f,Z)$. Main theorem \ref{mainthm1} tells us that there exists a finite extension $K''$ of $K$ such that $\moi_{K''}(f,Z)=\moi_{K''}(g,Z)$. We now use Proposition \ref{Implication} again to obtain that $\moi_{K''}(f,Z)=\moi_{K''}(g,Z)=\moi_K(g,Z)$.

On the other hand, by the definition of the motivic oscillation index, we have $\moi_{K''}(f,Z)\geq \moi_K(f,Z)$. Therefore we have $\moi_K(f,Z)=\moi_K(g,Z)$ and Conjecture \ref{highest} holds for $(f,Z)$.
\end{proof}

\subsection{Main Theorem \ref{mainthm2} and its corollaries}

The highest weight satisfies a Thom-Sebastiani type property.

\begin{prop}[Thom-Sebastiani for highest weight]\label{Thom-Sebastiani}   Let $f_1\in\overline{\QQ}[x_1,\dots,x_{n_1}]$, $f_2\in \overline{\QQ}[y_1,\dots,y_{n_2}]$ and   $f=f_1+f_2$. Let $Z_1$ and $Z_2$ be subschemes of  $\AA_{\overline{\QQ}}^{n_1}$ and $\AA_{\overline{\QQ}}^{n_2}$, respectively, and $Z=Z_1\times Z_2$.
Let $\cO$ be a finitely generated $\ZZ$-subalgebra of $\overline{\QQ}$ over which $f_1,f_2, Z_1,Z_2$ are defined, and let $K\supset \cO$ be a number field.
 Then for all $m\geq 1$ and all $L\in\cL'_{K,1}$ we have
$$w(f_{k_L},b,Z_{k_L}^{(m-1)})=w(f_{1,k_L},b,Z_{1,k_L}^{(m-1)})+w(f_{2,k_L},b,Z_{2,k_L}^{(m-1)})$$
for all $(b_0, b_{m-1})\in k_L \times k_L^*$.
\end{prop}

\begin{proof} By construction we have $Z_{k_L}^{(m-1)}=Z_{1,k_L}^{(m-1)}\times Z_{2,k_L}^{(m-1)}$ and $\cL_{f_{k_L},b,\Psi}=\phi_1^*(\cL_{f_{1,k_L},b,\Psi})\otimes \phi_2^*(\cL_{f_{2,k_L},b,\Psi})$, where $\phi_1$ and $\phi_2$ are the projections from $Z_{k_L}^{(m-1)}$ to
$Z_{1,k_L}^{(m-1)}$ and $Z_{2,k_L}^{(m-1)}$, respectively. The claim follows by the K\"unneth formula.
\end{proof}

\begin{cor}\label{Th-S} Using the notation of Proposition \ref{Thom-Sebastiani}, suppose that $f_1(Z_1(\CC))$ contains only one critical value $c$ of $f$ and that $f_1-c$ has non-rational singularities in every neighborhood of $Z_1(\CC)$ in $\CC^{n_1}$. If Conjecture \ref{highest} holds for $(f_1,Z_1)$ and $(f_2,Z_2)$ then it also holds for $(f,Z)$.
\end{cor}

\begin{proof}
Combining our assumption with Theorem  \ref{goodreduction}, there exists an integer $M$ such that $-\moi_K(f_1,Z_1)=-\lct_{Z_1}(f_1-c)$ is always the real part of a non-trivial pole of $Z(f_1-c,\chi_{\textnormal{triv}},L,\Phi_{L,Z_1},s)$ for all local fields $L\in \cL_{K',M}$, where $K'$ is some finite extension of $K$. By \cite[Section 5.1]{DenefBourbaki} and the definition of motivic oscillation index, we then have $\moi_K(f,Z)=\moi_K(f_1,Z_1)+\moi_K(f_2,Z_2)$. Our claim follows by combining this fact with Proposition \ref{Thom-Sebastiani}.
\end{proof}

\begin{proof}[Proof of Main Theorem \ref{mainthm2}] The claim follows by combining Theorems \ref{nonrational} and \ref{transferanalytic} and Corollary \ref{Th-S}.
\end{proof}

\begin{proof}[Proof of Corollary \ref{ADE}] If $f$ has singularities of type $A-D-E$ at $P$, then $(f,P)\sim_K (g,0)$, where $g$ is one of the following polynomials:
\begin{itemize}
\item[(i)]$g=x_1^{d+1}+x_2^2+\dots+x_n^2$ \quad with $d\geq 1$,
\item[(ii)]$g=x_1^{d-1}+x_1x_2^2+x_3^2+\dots+x_n^2$ \quad with $d\geq 4$,
\item[(iii)]$g=x_1^3+x_2^4+x_3^2+\dots+x_n^2$,
\item[(iv)]$g=x_1^3+x_1x_2^3+x_3^2+\dots+x_n^2$,
\item[(v)]$g=x_1^3+x_2^5+x_3^2+\dots+x_n^2$,
\end{itemize}
where $n\geq 3$. Since $x^d$ (with $d>1$), $x_1^3+x_1x_2^3$ and $x_1^{d-1}+x_1x_2^2$ have non-rational singularities at $0$, our claim follows by Theorem \ref{transferanalytic} and Proposition \ref{Implication}.
\end{proof}

\begin{proof}[Proof of Corollary \ref{3variables}]
Let $f$ be a polynomial in three variables.  By Proposition \ref{reduction}  and the definition of $\moi_K(f,Z)$, it suffices to prove our theorem in case $f(Z(\CC))=0$. If $f$ is smooth at $Z_\CC$,  then there exists $M$ such that $E_{f,L,Z,\psi}=0$ for all $L\in\cL_{K,M}$ and $m_{\psi}\geq 2$ (see \cite[Remark 4.5.3]{DenefBourbaki} or by using Proposition \ref{Implication} and Lemma \ref{smooth}), and then our claim follows from the fact that $\moi_K(f,Z)=+\infty$. Denote by $\Sing(f)$ the singular locus of $f$, we derive from the smooth case that $E_{f,L,Z,\psi}=E_{f,L,Z\cap\Sing(f),\psi}$ if $m_{\psi}\geq 2$, and hence we can reduce to the case $Z\subset \Sing(f)$.

If $\dim(Z)>0$, then $f$ has non-rational singularities in every neighbourhood of $Z(\CC)$, so our claim follows by \cite{CMN}. If $\dim(Z)=0$, we can reduce to the case $Z=\{0\}$. If $f$ has a non-rational singularity at $0$, we are again done by \cite{CMN}. If $f$ is non-smooth and has a rational singularity at $0$, then $f$ has singularity of type $A-D-E$ (see \cite{Durfee}). This case follows by Corollary \ref{ADE}.
\end{proof}

\section{Appendix}\label{section7}

Let $L$ be a $p$-adic field. To a non-constant polynomial $f\in L[x_1,\dots,x_n]$, a  Schwartz-Bruhat function $\Phi$ on $L^n$, and a character $\chi$ of $\cO_L^*$, we associated the Igusa local zeta function
$Z(f,\chi,L,\Phi;s)$.
The maximal possible order of its poles is $n$.

When $\chi$ is the trivial character, Laeremans and the second author \cite{Laeremans-Veys} showed that the real part a pole of order $n$ is always of the form $1/N$ for some positive integer $N$.
They also conjectured that there is at most one possible real part of a pole or order $n$, and that it is then necessarily equal to $-\lct_{Z}(f)$.
This was shown by Nicaise and Xu in \cite{Nicaise-Xu}.

When $\chi$ is a non-trivial character of order $d>1$, Cluckers, Musta\c t\u a and the first author asked in \cite{CMN} the following. {\em If $Z$ is a subscheme of $\AA_{\cO_L}^n$ and $s_0$ is the real part of a pole of order $n$ of $Z(f,\chi,L,\Phi_{Z},s)$, does this imply that $s_0=-\lct_{Z}(f)$ and $s_0=-1/dk$ for some positive integer $k$?}
They obtain in this direction the following result \cite[Proposition 3.8]{CMN}: with $f$ and $Z$  as in Notation \ref{notation2}, if $s_0=-\lct_{Z}(f)$ is a pole of order $n$ of $Z(f,\chi,L,\Phi_Z,s)$ for infinitely many $L$ with arbitrarily large residue field characteristic, then $\lct_{Z}(f) \leq 1/d$.

Here we answer their question affirmatively.

\begin{thm}
Let $L$ be a $p$-adic field, $f\in L[x_1,\dots,x_n]\setminus L$, and $\Phi$ a Schwartz-Bruhat function on $L^n$. Let $\chi$ be a character of $\cO_L^*$ of order $d>1$.
If $s_0$ is the real part of a pole of order $n$ of $Z(f,\chi,L,\Phi;s)$, then $s_0=-\lct_{Z}(f)=-1/dk$ for some positive integer $k$.
\end{thm}

\begin{proof}
If $s_0$ is the real part of a pole of order $n$ of $Z(f,\chi,L,\Phi;s)$, then for any embedded resolution $h:Y\to \AA_L^n$ of $f^{-1}(0)$, there must exist $n$ different components $E_1,\dots,E_n$ of $h^{-1}(f^{-1}(0))$, satisfying the following, where we use notation as in subsection \ref{Denef's formula}.
There exists a point $P \in \cap_{i=1}^n E_i$ such that $h(P)\in Z$,  $s_0=  -\nu_1/N_1 = \dots = - \nu_n/N_n$,  and $d \mid N_i$ for  $i=1,\dots,n$.

This implies that $s_0= - \lct_{Z}(f)$ by the results in  \cite{Nicaise-Xu}. The second statement follows from the geometric result in Proposition \ref{geometric} below.
\end{proof}

\begin{prop}\label{geometric} Let $F$ be an algebraically closed field of characteristic zero.
 Let $D = \sum_i  N_iD_i$ be an effective divisor on a nonsingular $F$-variety $Y$ of
dimension $n$. Take an embedded resolution $h : X \to Y$  of $D$, constructed as a finite composition of admissible blow-ups, and let $E_i, i\in T,$ be the irreducible components of $(h^{-1}(D))_{\operatorname{red}}$. For each $i\in T$, let $N_i$ and $\nu_i -1$ be the multiplicities of $E_i$ in the divisor of $f\circ h$ and in the relative canonical divisor $K_h$, respectively.

Let $d$ be a positive integer.  Suppose
that there exist $n$ different $E_i, i\in I\subset T,$ such that $\cap_{i\in I} E_i \neq \emptyset$, all $\nu_i/N_i, i\in I,$  are equal, and $d \mid N_i$ for all $i\in I$. Then there exists a positive integer $k$ such that  $\nu_i/N_i=1/dk$ for $i\in I$.
\end{prop}

\begin{proof} We adapt the argument of  \cite{Laeremans-Veys}. For completeness, we mention most details.

 Denote $r=\nu_i/N_i$ for $i\in I$.
If at least one $E_i, i \in I,$ is an irreducible component of the strict transform of
$D$, then clearly $r$ is $1/N_i$ with $d\mid N_i$. So from now on we suppose that all $E_i, i\in T,$  are exceptional
varieties.
At a certain step of the resolution process h one of the $E_i, i\in I,$  is created as the
exceptional variety of a blow-up and the other ones are strict transforms of previously
created exceptional varieties. We now consider the following situation ($\dagger$) of which this
step is a special case.

\begin{itemize}
\item[($\dagger$)]
Let $h_0 : X_1 \to X_0$ be a blow-up of $h$ with centre $C_0$ of codimension $c\geq 2$ in
$X_0$ and exceptional variety  $E_1 \subset X_1$. Suppose that
 there exists a point $P \in E_1$ belonging to $n$ different exceptional varieties of
$h$, say $P \in E_1 \cap \tilde{E}_2 \cap \dots \cap \tilde{E}_n$, where $\tilde{E}_j$ is the strict transform of $E_j\subset X_0$ for
$j = 2, \dots, n$, such that
\begin{itemize}
\item[(i)] there exist $a_1,a_2, \dots, a_n \in \ZZ$ satisfying
$$\frac{\nu_1}{N_1+a_1}=\frac{\nu_2}{N_2+a_2}=\dots=\frac{\nu_n}{N_n+a_n}=r \quad \text{(and all $N_i+a_i \neq 0$)} ,$$
    \item[(ii)] $ d \mid N_i+a_i$ \quad for $i=1,\dots,n$.
\end{itemize}
\end{itemize}

\noindent
Since $C_0$ has normal crossings with $\cup_{i=2}^n E_i$, it is easy to check that (after renumbering)
 $E_2, \dots, E_c \supset C_0$ and $E_{c+1}, \dots, E_n \not\supset C_0$.
So we are left with two possibilities:

(1) no other exceptional variety of $h$ contains $C_0$, or (say)

(2) also $E_{n+1} \supset C_0$.

\noindent
We claim that then

(1') $r = 1/dk$ for some positive integer $k$, and

(2') $r=  \frac{\nu_{n+1}}{N_{n+1}+a_{n+1}}$ for some $a_{n+1} \in \ZZ$, satisfying $d\mid  N_{n+1}+a_{n+1}$,

\noindent
respectively. Recall now the well-known equalities
$$N_1 = \sum_{i=2}^c N_i + \mu \  [+N_{n+1}] \quad\text{ and } \quad \nu_1 = \sum_{i=2}^c (\nu_i-1) + c \ [+\nu_{n+1}-1],
$$
where $\mu$ is the multiplicity of the generic point of $C_0$ on the strict transform of $D$ on
$X_0$, and in case (1) and (2) the terms within square brackets do not and do occur, respectively.
So
$$
r= \frac{\nu_1}{N_1+a_1} = \frac{\sum_{i=2}^c \nu_i + 1 \ [+\nu_{n+1}-1]}{\sum_{i=2}^c (N_i+a_i) + \mu + (a_1-\sum_{i=2}^c a_i ) \  [+N_{n+1}]},
$$
which implies that
$$
r= \frac{ 1 \ [+\nu_{n+1}-1]} {\mu + a_1-\sum_{i=2}^c a_i \  [+N_{n+1}]}.
$$
In case (1) we are done since $d$ then divides $\mu + a_1-\sum_{i=2}^c a_i = N_1+a_1 - \sum_{i=2}^{c} (N_i+a_i)$.
In case (2) we define  $a_{n+1}:= \mu + a_1-\sum_{i=2}^c a_i$. Then $r$ is of the stated form and, similarly, $d$ divides $N_{n+1}+a_{n+1}$. This proves our claim.

Now we can derive the statement in  the theorem by consecutive applications of our study of the situation ($\dagger$). Start with the blow-up of $h$ where $E_1$ is created. In case (1) we are done. In case (2) we obtain (using the notation above) $Q \in \cap_2^{n+1}E_i$ which induces by
(2') a new situation ($\dagger$). We can now repeat the same arguments until, by finiteness of
the resolution process, we encounter a case (1).
\end{proof}

\end{document}